\documentclass[a4paper]{article}
\usepackage{mystyle}

\begin{document}

\maketitle

\begin{abstract}
  Reduced basis approximations of Optimal Control Problems (OCPs) governed by steady partial differential equations (PDEs) with random parametric inputs are analyzed and constructed. Such approximations are based on a Reduced Order Model, which in this work is constructed using the method of weighted Proper Orthogonal Decomposition. This Reduced Order Model then is used to efficiently compute the reduced basis approximation for any outcome of the random parameter. We demonstrate that such OCPs are well-posed by applying the adjoint approach, which also works in the presence of admissibility constraints and in the case of non linear-quadratic OCPs, and thus is more general than the conventional Lagrangian approach. We also show that a step in the construction of these Reduced Order Models, known as the aggregation step, is not fundamental and can in principle be skipped for noncoercive problems, leading to a cheaper online phase. Numerical applications in three scenarios from environmental science are considered, in which the governing PDE is steady and the control is distributed. Various parameter distributions are taken, and several implementations of the weighted Proper Orthogonal Decomposition are compared by choosing different quadrature rules.
\end{abstract}
 
\section{Introduction}

The search for a solution to a PDE-constrained optimization problem is in practice \gcok{affected} by unavoidable uncertainties, \reviewerB{associated with} measurements of parameters involved in the optimization problem. In such contexts of Uncertainty Quantification, one models these parameters as random variables. If the solution is a measurable function of the \reviewerB{parameters}, then it is a random variable itself, and one can study statistics that depend on the solution. For example, one may be interested in a moment of the solution or of a measurable function of the solution. In order to estimate such moments, a Monte Carlo type of estimator could be constructed, which takes the average over a large amount of solutions of the optimization problem, each corresponding to some outcome of the random parameter. When these solutions are not explicitly available, they can be approximated with accurate yet computationally expensive methods, such as Finite Element based methods. Consequently, the use of these so-called truth approximations in the construction of the estimator \gcok{can lead to significant computational costs}. To accommodate this issue, a Reduced Order Model (ROM) can be used to accelerate the approximation process, by providing a cheaply computable surrogate of the expensive truth approximation for any given parameter value, called a reduced \gcok{basis} approximation. The interested reader may refer to \gcok{\cite{rozza2016, prudhomme01, huynh13, rozza2008}} for a survey on ROM techniques and to \cite{bader2017, dede2010reduced, karcher2018certified, negri15, negri13} for their application to parametrized Optimal Control Problems, to \gcok{\cite{torlo18, venturi19, venturitorlo19}} for their application to UQ and to \cite{chen17} for the application to OCPs in UQ.
\newline
Given a parametric measurement $\mu\in\R^n$ for $n\in\N$, in the general formulation of an Optimal Control Problem (OCP) parametrized by $\mu$, one is to minimize a convex functional $J(\cdot,\cdot;\mu):Y\times U\rightarrow\R$ over all state-control pairs $(y,u)\in Y\times U$ that satisfy the governing PDE-state equation $e(y,u;\mu)=0$. Here, $Y$ and $U$ are the real Hilbert spaces of state and control, respectively, and $e(\cdot,\cdot;\mu):Y\times U\rightarrow Q$ with $Q$ some real Hilbert space. The minimizer $(y(\mu),u(\mu))$ is often constrained to lie in a closed and convex set of admissible pairs $W_{ad}$ of $Y\times U$. Our focus, however, mainly lies on OCPs with $W_{ad}=Y\times U$, as it allows for the use of a ROM, and with $J(\cdot,\cdot;\mu)$ of quadratic form.
\newline
It is well-known (\gcok{see e.g.} \cite{lions71}), that the minimizer $(y(\mu),u(\mu))$ can be obtained by finding the solution $\left( y(\mu),u(\mu),p(\mu) \right)\in Y\times U\times Q^*$ to some system of three equations, as shall also be recalled in this work. The variable $p(\mu)$ here is called the adjoint solution.
To incorporate uncertainty in the parametric measurement $\mu$, we interpret $\mu$ as the outcome of a random variable that takes values in some subset $\m\subset\R^n$. For any outcome $\mu\in\m$, the solution $(y(\mu),u(\mu),p(\mu))$ is in general not explicitly available, and one could construct the computationally expensive truth solution $(y^\nd(\mu),u^\nd(\mu),p^\nd(\mu))$ to approximate it. As indicated above, the aim is to solve the OCP for a large number of possible outcomes of this random variable, so that, for example, sample averages can be computed. Not only would a naive parameter wise computation of the truth approximation become computationally infeasible, it would also completely ignore any possible low-dimensional behaviour of the (discrete) solution manifold
\[
  \M^\nd:=\left\{ (y^\nd(\mu),u^\nd(\mu),p^\nd(\mu)),\ \mu\in\m \right\}\subset Y\times U\times Q^*.
\]
Indeed, the Kolmogorov $\nr$-width
\[
  \inf_{\substack{V\subset\Span{\M^\nd}\\\dim{V}=\nr}}\esssup_{\mu\in\m}\norm{(y^\nd(\mu),u^\nd(\mu),p^\nd(\mu))-P_V(y^\nd(\mu),u^\nd(\mu),p^\nd(\mu))}_{Y\times U\times Q^*}
\]
where $P_V$ denotes the orthogonal projector onto $V$, very often decays rapidly in $\nr$. %(TODO: ref some success papers of ROM).
A ROM aims to exploit this by constructing a low dimensional subspace of $\Span{\M^\nd}$, in a computationally expensive so-called offline phase, onto which for any given outcome $\mu\in\m$ the solution $\left( y^\nd(\mu),u^\nd(\mu),p^\nd(\mu) \right)$ then can be projected. This projection is performed in the online phase, which under certain parameter separability conditions can be executed rapidly, and results in a reduced \gcok{basis} approximation $(y_\nr(\mu),u_\nr(\mu),p_\nr(\mu))$ of \gcok{the function} $(y(\mu),u(\mu),p(\mu))$. The low dimensional subspace should be constructed in some optimal way.  In this work, we do this by means of the weighted Proper Orthogonal Decomposition algorithm, already proposed in \gcok{\cite{venturi19}}. This algorithm is a combination of a singular value decomposition and a quadrature rule. \gcok{See also \cite{ballarin15, burkardt06, chapelle13}.}
\newline
Clearly, before one can construct truth approximations and reduced \gcok{basis} approximations of the control problem corresponding to an outcome $\mu\in\m$, it must be guaranteed that the OCP is well-posed for that $\mu$. One way of doing this, is by using Brezzi's Theorem to search for saddle points of the Lagrangian corresponding to the OCP, as in \cite{negri15, negri13, maria17}. We shall present a second approach, also known as the adjoint approach or the Lions approach, based on the argumentation of \cite[Chapters~I, II]{lions71} and \cite[Chapter~I]{hinze09}. It allows one to recover the conclusions of the Lagrangian approach, but in a more general form, as it does not require $e(\cdot,\cdot;\mu)$ to be linear, $J(\cdot,\cdot;\mu)$ to be quadratic and $W_{ad}$ to be $Y\times U$.
\newline
The work is outlined as follows. First of all, we formally define the formulation of our OCP in an Uncertainty Quantification context in Section \ref{sec:problem_formulation}. We then briefly recall the main arguments of the adjoint approach in Section \ref{sec:ocp}, where the discussion is generalized from Hilbert to Banach spaces. For convenience, all linear spaces are assumed to be real.
\newline
The well-posedness of the truth approximations and reduced \gcok{basis} approximations is derived with similar argumentation as for the original OCP, as argued in Sections \ref{sec:truth_approximation} and \ref{sec:rom}, respectively. For the reduced formulation, recent developments of ROMs for parametrized OCPs have made use of so-called aggregate spaces to ensure that the reduced formulation is well-posed, \cite{bader2017, dede2010reduced, karcher2018certified, negri15, negri13}. We shall argue that this result can be improved by showing that the aggregation procedure is actually redundant from a theoretical point of view, and only useful in the specific case of coercive governing PDEs. This leads to an additional acceleration in the performance of the ROM, as the low dimensional space constructed in the offline phase can be taken even smaller. Section \ref{sec:rom} also recalls the basic ideas of ROMs and the weighted Proper Orthogonal Decomposition.
\newline
Finally, in Section \ref{sec:applications}, we reconstruct the three applications of \cite{maria17} in marine sciences. Parametric uncertainties are inherent to such real world applications, and it therefore is important to embed them in the Uncertainty Quantification context. We apply the weighted Proper Orthogonal Decomposition method with different choices of the quadrature rule to construct ROMs that accurately and efficiently approximate the OCP for arbitrary draws of the random parameter. We shall also consider several compactly supported probability distributions on $\m$ from which these draws are taken. The first numerical application involves a coercive PDE in weak formulation. The other two involve noncoercive (weakly coercive) PDEs, and for these we compare the performance of ROMs obtained with and without the aggregation procedure. Conclusions follow in Section \ref{sec:conclusions}.
\gcok{
\newline
The main novelty of the work lies in the derivation of well-posedness of ROMs for OCPs via the adjoint approach including their associated validity on Banach spaces, nonreflexive state spaces, and non-aggregated spaces, as well as the numerical application of these ROMs to existing OCP models embedded in a UQ context.
}

\section{Problem formulation}
\label{sec:problem_formulation}
Let $(\Theta,\F,\P)$ be a probability space and $\m$ be a compact subset of $\R^n$ for $n\in\N$. Let also $\bm{\mu}:(\Theta,\F,\P)\rightarrow \m$ be a random variable with respect to the Borel $\sigma$-algebra on $\m$. We denote its law, i.e. the push forward measure of $\P$ under $\bm{\mu}$, by $\P^{\bm{\mu}}$. Furthermore, we let the Hilbert spaces $Y$ and $U$ be the \textit{state space} and \textit{control space} respectively and let $J(\cdot,\cdot;\cdot):Y\times U\times\m\rightarrow\R$ denote the \textit{objective functional}. Finally, let $e(\cdot,\cdot;\cdot):Y\times U\times \m\rightarrow Q$ be the \textit{state equation function} for some Hilbert space $Q$. We are interested in solving the following parametrized OCP:
 \begin{problem}
   \gcok{Let us find} $(y(\bm{\mu}),u(\bm{\mu})):\Theta\rightarrow Y\times U$ such that it $\P$-almost everywhere holds that
   \begin{itemize}
     \item $e(y(\bm{\mu}),u(\bm{\mu});\bm{\mu})=0$,
     \item $(y(\bm{\mu}),u(\bm{\mu}))\in W_{ad}\subset Y\times U$,
     \item $J(y(\bm{\mu}),u(\bm{\mu})\reviewerB{;\bm{\mu}})=\min_{(\tilde{y},\tilde{u})\in Y\times U}J(\tilde{y},\tilde{u};\bm{\mu})$
   \end{itemize}
  \label{prob:general_objective}
\end{problem}
If $\P^{\bm{\mu}}$ corresponds to the uniform distribution on $\m$, then we say that the parametrized OCP is \textit{deterministic}.
We are primarily interested in the following particular kind of parametrized OCP. By $\B(H_1, H_2)$ we denote the \gcok{space of bounded linear} operators between the normed spaces $H_1$ and $H_2$. 

\begin{definition}
  Let $Z$ be a Hilbert space, the so-called \textit{observation space}, $z_d(\bm{\mu}):\Theta\rightarrow Z$ a \textit{desired} solution profile, $C\in\B(Y,Z)$ an \textit{observation operator}.
  Define the quadratic objective functional $J(\cdot,\cdot;\bm{\mu})$
\begin{align}
  J(y,u;\bm{\mu}) = \frac{1}{2}\langle M(\bm{\mu})(Cy-z_d(\bm{\mu})), Cy-z_d(\bm{\mu})\rangle_{Z^*Z} + \frac{1}{2}\langle \Q(\bm{\mu})u,u\rangle_{U^*U},
  \label{eqn:quadratic_objective}
\end{align}
with $M(\bm{\mu}):\Theta\rightarrow\B(Z,Z^*),\  \Q(\bm{\mu}):\Theta\rightarrow\B(U,U^*)$ both self-adjoint.
\newline
Let $e(\cdot,\cdot,\bm{\mu})$ be affine, that is, there exist $A(\bm{\mu}):\Theta\rightarrow\B(Y,Q),\ B(\bm{\mu}):\Theta\rightarrow\B(U,Q),\ g(\bm{\mu}):\Theta\rightarrow Q$ such that 
\begin{align}
  e(y,u;\bm{\mu}) = A(\bm{\mu})y+B(\bm{\mu})u-g(\bm{\mu})\hspace{5mm} \forall y\in Y, u\in U.
  \label{eqn:affine_e}
\end{align}
Then the corresponding Problem \ref{prob:general_objective} is called a \textit{linear-quadratic} OCP.
\label{def:linear_quadratic}
\end{definition}

\section{Solutions of Optimal Control Problems}
\label{sec:ocp}
In this section we describe the well-posedness of an OCP, temporarily dropping the parameter from the notation. The description is generalized from Hilbert to Banach spaces $Y,U$ and $Q$. Throughout this section, we make the following assumption.

\begin{assumption}
  It holds that
      \begin{enumerate}
	\item $W_{ad}=Y\times U_{ad}$ for $U_{ad}\subset U$,
	\item for every $u\in U_{ad}$ there exists a unique $y=y(u)\in Y$ \gcok{with} $e(y,u)=0$.
      \end{enumerate}
  \label{ass:differentiability2}
\end{assumption}

While the first assumption is made for the sake of simplicity, the second assumption is fundamental for the use of the adjoint approach, as it allows us to consider the following (unparametrized) reformulation of Problem \ref{prob:general_objective}:

\begin{problem}
  \gcok{Let us minimize} $\tilde{J}(u):=J(y(u),u)$ over all $u\in U_{ad}$.
  \label{prob:reduced_objective}
\end{problem}

Existence and uniqueness of Problem \ref{prob:reduced_objective} is discussed in the following proposition. If $\J$ is G-differentiable (Gateaux differentiable) in $u$, then we denote the G-derivative in $u$ by $D\J(u)\in\B(U,U^*)$.

\begin{proposition}
  \gcok{Let us suppose} that Assumption \ref{ass:differentiability2} holds and that $U_{ad}$ is nonempty and convex.
  \begin{enumerate}[label=(\roman*)]
    \item For \gcok{the} uniqueness of \gcok{the} solution of Problem \ref{prob:reduced_objective}, it suffices that $\J$ is strictly convex.
    \item For \gcok{the existence of a} solution of Problem \ref{prob:reduced_objective}, the following set of conditions is sufficient: $U$ is reflexive, ${U}_{ad}$ is closed, $\J$ is weakly lower semicontinuous and $\J(u)\rightarrow\infty$ whenever $\norm{u}_U\rightarrow\infty$ in $ U_{ad}$.
    \item If $u\in U_{ad}$ is a local optimizer, and if $\J$ is {\upshape G}-differentiable in $u$, then $u$ satisfies
      	\begin{align}
	  \langle D\J(u),\tilde{u}-u\rangle_{U^*U}\geq 0\hspace{5mm} \forall \tilde{u}\in  U_{ad}.
  		\label{eqn:optimality}
  	\end{align}
	If $\J$ is convex and (\ref{eqn:optimality}) holds for some $u\in U_{ad}$, then $u$ is a global optimizer.
      \item For \gcok{the existence of} a solution of Problem \ref{prob:reduced_objective}, the following conditions are sufficient:\\
  $\J$ is convex and {\upshape G}-differentiable in a neighbourhood of $U_{ad}$, and (\ref{eqn:optimality}) holds for some $u\in U_{ad}$.
 \end{enumerate}
  \label{prop:general_criterion}
\end{proposition}

\begin{proof}
  See Subsections 1.2, 1.3 in Chapter 1 of \cite{lions71}, \gcok{Theorem} 1.46 in \cite{hinze09}.
\end{proof}

By explicitly computing the derivative of $\tilde{J}$, these statements lead to the following conclusion on well-posedness of Problem \ref{prob:reduced_objective}.

\begin{theorem}
  Suppose that Assumption \ref{ass:differentiability2} holds, $J$ and $e$ are continuously {\upshape F}-differentiable and $e_y(y(u),u)\in \B(Y,Z)$ has a bounded inverse for each $u\in U$. Then $\J$ is {\upshape F}-differentiable.  Suppose further that
  \begin{enumerate}
    \item $U$ is reflexive,
    \item $\J$ is strictly convex and $\J(u)\rightarrow\infty$ when $\norm{u}_U\rightarrow\infty$ in ${U}_{ad}$,
    \item ${U}_{ad}$ is nonempty, closed and convex.
  \end{enumerate}
  Then there exists a unique solution $u\in U_{ad}$ to Problem \ref{prob:reduced_objective} and it obeys, for some unique $p(u)\in Q^*$,
  \begin{equation}
    \begin{aligned}
      e(y(u),u) &=0  &\text{in}\ Q, \\
      e_y(y(u),u)^*p(u)+J_y(y(u),u)&=0 &\text{in}\ Y^*,\\
      \langle e_u(y(u),u)^*p(u) + J_u(y(u),u),\tilde{u}\rangle_{U^*U}&\geq 0 &\forall \tilde{u}\in U_{ad}.
    \end{aligned}
      \label{eqn:general_optimizing_equations}
  \end{equation}
  \label{thm:constrained_opt_eqns}
\end{theorem}

\begin{proof}
  The proof is based on an application of the Implicit Function Theorem. See Section 1 in Chapter 2 of \cite{lions71}, or Sections 1.5.1 and 1.6.2 of \cite{hinze09}.
\end{proof}

\begin{remark}
  The conclusion of Theorem \ref{thm:constrained_opt_eqns} also holds for more general forms of $W_{ad}$ and in the case the maps $J$ and $e$ are only F-differentiable in a neighbourhood of $U_{ad}$. See \cite[Section 1.7]{hinze09} and \cite[Section 1.3]{carere20}.
\end{remark}

The first equation of (\ref{eqn:general_optimizing_equations}) is again the state equation, while the second and third are known as the \textit{adjoint equation} and \textit{optimality criterion}, respectively.
In the case of a linear-quadratic problem, we can now say the following.

\begin{corollary}
  Suppose $U_{ad}$ is nonempty, closed and convex and that $U$ is reflexive. In the case of a linear-quadratic problem, in which $J$ and $e$ are of the form (\ref{eqn:quadratic_objective}) and (\ref{eqn:affine_e}), respectively, with $A$ boundedly invertible, $M(z,z)\geq 0$ for $z\in Z$, and $\Q$ coercive, there exists a unique minimizer $u\in U_{ad}$ of Problem \ref{prob:reduced_objective} such that, for some unique $p(u)\in Q^*$,
  \begin{equation}
  \begin{aligned}
    Ay(u)+Bu				&=g		&\text{\ in\ } Q,\\
    C^*MCy(u)+A^*p(u) 			&= C^*Mz_d 	&\text{\ in\ } Y^*,\\
    \langle \Q u+B^*p(u),\tilde{u}-u\rangle_{U^*U} 	&\geq 0  	&\forall \tilde{u}\in  U_{ad}.
  \end{aligned}
  \label{eqn:gen_lin_qdr_opt}
\end{equation}
  \label{cor:general_linear_quadratic}
\end{corollary}

\begin{proof}
  For $J$ being of quadratic form, F-differentiability is immediate. Its partial derivates can easily be computed, as $M$ and $\Q$ are self-adjoint, to be $\langle J_u(y,u),\tilde{u}\rangle_{U^*U}= \langle \Q\tilde{u},u\rangle_{U^*U}$ and $\langle J_y(y,u),\tilde{y}\rangle_{Y^*Y} = \langle MC\tilde{y},Cy\rangle_{Z^*Z}-\langle MC\tilde{y},z_d\rangle_{Z^*Z}$, for $\tilde{u}\in U, \tilde{y}\in Y$. Similarly, $e$ is F-differentiable with $e_y=A,e_u=B$. We see that all partial derivatives are continuous, so that $J$ and $e$ are continuously F-differentiable.
Clearly, $y(u)=A^{-1}(g-Bu)$ and Assumption \ref{ass:differentiability2} holds. The conditions on $M$ and $\Q$ imply $\J(u)\geq \lambda \norm{u}_U^2$, where $\lambda$ denotes the coercivity constant of $\Q$. Hence $\J(u)\rightarrow\infty$ as $u\rightarrow\infty$. By the same assumption on $\Q$ it follows that $u\mapsto \langle \Q u,u\rangle_{U^*U}$ is strictly convex. Furthermore, $u\mapsto \langle M(Cy(u)-z_d),Cy(u)-z_d\rangle_{Z^*Z}$ is convex by the assumption on $M$ and the affinity of $e$. Strict convexity of $\J$ follows.
The result now follows from (\ref{eqn:general_optimizing_equations}) by substitution.
\end{proof}

\begin{remark}
  \gcok{The assumptions for well-posedness usually considered in \reviewerB{the} literature often include that $Y=Q^*$ and $A$ is coercive, i.e. $\sup_{0\not=y\in Y}\langle Ay,y\rangle_{QQ^*}>0$, but well-posedness still holds if the first is omitted and if $A$ is merely boundedly invertible.}
  Note also that reflexivity of $Y$ and $Q$ is not required.
\end{remark}

So to solve a linear-quadratic OCP under the conditions of Corollary \ref{cor:general_linear_quadratic}, we must find the triple $(y,u,p)\in Y\times U\times Q^*$ that solves the system (\ref{eqn:gen_lin_qdr_opt}). In the \textit{full admissibility} case, that is, in the case $W_{ad}=Y\times U$, the third equation becomes $\Q u+B^*p = 0$ in $U^*$. Assuming also that $Q$ is reflexive, so that $Q=P^*$ with $P:=Q^*$, system (\ref{eqn:gen_lin_qdr_opt}) then can be written as
\begin{equation}
  \begin{aligned}
    &C^*MCy(u)	& & 	&+&A^*p(u) 	&	&= C^*Mz_d 	& &\text{ in } Y^*,\\
    & 		& &\Q u	&+&B^*p(u) 	&	&=0 		& &\text{ in } U^*,\\
    &Ay(u)	&+&Bu 	& &		&	&=g		& &\text{ in } P^*.
  \end{aligned}
  \label{eqn:lin_qdr_opt}
\end{equation}
The Banach space $P$ is called the \textit{adjoint space}, \gcok{and} $p(u)$ the \textit{adjoint solution}.

\begin{remark}
  The obtained conclusions for linear-quadratic problems with full admissibility can also be recovered by studying saddle points of a Lagrangian, see e.g. \cite{karcher14}. For a comparison between the Lagrangian approach and the adjoint approach presented here, see \cite{carere20}.
\end{remark}

\section{Truth Approximations}
\label{sec:truth_approximation}
We present a standard approximation procedure for linear-quadratic OCPs with full admissibility based on Galerkin Projection. We assume that $U$ and $Q$ are reflexive and write $P=Q^*$. We also assume that $\Q(\bm{\mu})$ is coercive almost everywhere and that $\langle M(\bm{\mu})z,z\rangle_{Z^*Z}\geq 0$ for $z\in Z$ almost everywhere.
\newline
\newline
Taking closed subspaces $Y^\nd\subset Y,\  U^\nd\subset U$ and $P^\nd\subset P$, we obtain for an outcome $\mu$ of $\bm{\mu}$, an approximation $(y^\nd(\mu),u^\nd(\mu),p^\nd(\mu))\in Y^\nd\times U^\nd\times P^\nd$ of $(y(\mu),u(\mu),p(\mu))$ by a Galerkin Projection of $(y(\mu),u(\mu),p(\mu))$ onto $Y^{\nd}\times U^\nd \times P^\nd$. That is, we solve (\ref{eqn:lin_qdr_opt}) with $Y,U$ and $P$ replaced by respectively $Y^\nd,U^\nd$ and $P^\nd$. Defining 
\begin{align*}
  A^\nd(\bm{\mu})&\in L^2(\Theta;\B(Y^\nd,(P^\nd)^*)), &\hspace{5mm} &\text{by}& A^\nd(\bm{\mu}) y^\nd &= (A(\bm{\mu})y^\nd)_{|P^\nd},\\
  B^\nd(\bm{\mu})&\in L^2(\Theta;\B(U^\nd,(P^\nd)^*)), &\hspace{5mm} &\text{by}& B^\nd(\bm{\mu}) u^\nd &= (B(\bm{\mu})u^\nd)_{|P^\nd}, 
\end{align*}
we notice that this amounts to solving, $\P^{\bm{\mu}}$-almost everywhere, the OCP
\begin{align*}
  \text{minimize } &J_{|Y^\nd\times U^\nd}(\tilde{y}^\nd,\tilde{u}^\nd;\mu) \text{ s.t. } A^\nd(\mu) \tilde{y}^\nd + B^\nd(\mu) \tilde{u}^\nd = (g(\mu))_{|P^\nd} \\
  &\text{over all } (\tilde{y}^\nd,\tilde{u}^\nd) \in Y^\nd\times U^\nd.
\end{align*}

This approximate problem is called the \textit{truth problem}. By Corollary \ref{cor:general_linear_quadratic}, it is well-posed if $A^\nd({\mu})$ is boundedly invertible, as all other conditions are inherited from the original OCP in the continuous formulation. In that case, the approximate solution $(y^\nd(\mu),u^\nd(\mu),p^\nd(\mu))$ is called the \textit{truth approximation}. In Section \ref{sec:applications} we shall take $Y^\nd,\ U^\nd$ and $P^\nd$ to be Finite Element spaces of dimension of order $O(\nd)$\gcok{, see e.g. \cite[Chapters~3,5,6]{quarteroni94}}.

\begin{remark}
  Just as for well-posedness of the original OCP it is not necessary for $Y$ and $P$ to be equal, it is not required for $Y^\nd$ and $P^\nd$ to be equal, as long as $A^\nd(\mu)$ is boundedly invertible.
\end{remark}

\section{Reduced Order Method}
\label{sec:rom}
In addition to the assumptions of Section \ref{sec:truth_approximation}, let us now assume that $Y$, $U$ and $Q$ are Hilbert spaces, so that also $P$ is a Hilbert space. If $\mu\in\m$ is an outcome of $\bm{\mu}$, then we could compute the expensive truth approximation $(y^\nd(\mu),u^\nd(\mu),p^\nd(\mu))$ by performing a Galerkin Projection of $(y(\mu),u(\mu),p(\mu))$ onto the subspace $Y^\nd\times U^\nd\times P^\nd$, which has a high dimension of order $O(\nd)$. This high-dimensionality is necessary to ensure the accuracy of the truth approximations. Since solving many truth problems becomes computationally infeasible, in this section we propose the construction of a cheap \textit{reduced \gcok{basis} approximation} $(y_\nr(\mu),u_\nr(\mu),p_\nr(\mu))$ using a ROM, \cite{haasdonk10, rozza2016}. We discuss the implementation of ROMs for OCPs as already presented in \cite{bader2017, dede2010reduced, karcher2018certified, negri15, negri13}.

\subsection{Reduced basis approximation}
Given $\mu\in\m$, the reduced \gcok{basis} approximation $(y_\nr(\mu),u_\nr(\mu),p_\nr(\mu))$ can be obtained by performing a Galerkin Projection of $(y^\nd(\mu), u^\nd(\mu), p^\nd(\mu))$ onto a subspace $\X_\nr$ of $Y^\nd \times U^\nd \times P^\nd$, which has a low dimension of order $O(\nr)$, with $\nr\ll\nd$. In this subsection we discuss the construction of such a space based on the approach known as \textit{weighted Proper Orthogonal Decomposition} (\textit{weighted POD}), or just \textit{Proper Orthogonal Decomposition} (\textit{POD}). The process of this construction is called the \textit{offline phase}, and the resulting low dimensional space is called the \textit{reduced basis space}. The act of performing the Galerkin Projection onto the reduced basis space for a given outcome $\mu$ of $\bm{\mu}$ is known as the \textit{online phase}.
\newline
\newline
Given such an outcome $\mu\in\m$, the solution to the OCP $(y(\mu),u(\mu),p(\mu))$ can be approximated by the truth solution $(y^\nd(\mu),u^\nd(\mu),p^\nd(\mu))$, which, in turn, can be approximated with the reduced \gcok{basis} approximation \gcok{denoted by} $(y_\nr(\mu),u_\nr(\mu),p_\nr(\mu))$. In order for the ROM to have any use, one should be able to construct arbitrarily precise truth approximations, uniformly in the parameter. The whole task then is to construct a ROM in such a way that it is effective, i.e. that the error in approximating the truth solution by the reduced \gcok{basis} approximation is also small, and that $(y_\nr(\mu),u_\nr(\mu),p_\nr(\mu)$ can be computed efficiently.

\subsection{Offline phase: weighted Proper Orthogonal Decomposition}
\label{subsec:offline}
Let us write $\X=Y\times U\times P$. \gcok{Assuming} $\chi^\nd(\bm{\mu}):=(y^\nd(\bm{\mu}),u^\nd(\bm{\mu}),p^\nd(\bm{\mu})\in L^2(\Theta;\X)$, in order to find a suitable low dimensional space of $\X^\nd:=Y^\nd\times U^\nd\times P^\nd$, one could choose to construct the subspace $\X_\nr\subset\X^\nd$ in such a way that it minimizes the expected squared error
\begin{align*}
  \mathbb{E}\norm{\chi^\nd(\mu)-P_V\chi^\nd(\mu)}_{\X}^2 = \int_{\m}^{}\norm{\chi^\nd(\mu)-P_V\chi^\nd(\mu)}_{\X}^2\d\P^{\bm{\mu}}(\mu),
\end{align*}
among all subspaces $V$ of $\X^\nd$ of dimension at most $\nr$. Here $P_V$ denotes the orthogonal projector onto $V$.
\newline
\newline
Recall that the (discrete) solution manifold $\M^\nd$ is defined as $\left\{ \chi^\nd(\mu), \mu\in\m \right\}$. Defining \gcok{$\mathcal{C}\in \B(\Span{\M^\nd},\Span{\M^\nd})$ to be the compact, self-adjoint and positive operator}
\begin{align}
\label{eqn:svd_opt}
\mathcal{C}v = \mathbb{E}\left[\langle v,\chi^\nd(\bm{\mu})\rangle_{\X}\,\chi^\nd(\bm{\mu})\right],
\end{align}
it is well known (e.g. \cite{schwab06,griebel18}) that $\X_\nr=\Span{\xi_1,\ldots,\xi_\nr}$, where $(\lambda_i,\xi_i)_i$ is the eigenvalue-eigenvector sequence of $\C$ in which the eigenvalues are ordered in a decreasing fashion. Notice that the eigenvalues are positive and can accumulate only in 0. The problem is that this expectation is in general not known. To circumvent this problem, the weighted POD method uses a quadrature rule to approximate it.
\newline
\newline
Let $\m_d:=\{\mu_1,\ldots,\mu_M\}\subset\m$ and $\{w_1,\ldots,w_M\}\subset \R$ denote, respectively, the nodes and weights of a quadrature rule for $\P^{\bm{\mu}}$. The subscript `$d$' stands for `discrete'. We admit only nonzero weights, and write $\chi_i$ for the $i$th \textit{snapshot} $\chi^\nd(\mu_i)$. Let us define $\mathcal{C}_d\in\B(\Span{\chi_1,\ldots,\chi_M},\Span{\chi_1,\ldots,\chi_M})$, the discrete counterpart of $\mathcal{C}$, as
\begin{align*}
  \mathcal{C}_d v = \sum_{i=1}^{M}w_i\langle v,\chi_i\rangle_{\X}\,\chi_i.
\end{align*}
This operator is compact and self-adjoint but positive definite only when all weights are positive. It is not difficult to show, however, that the space $\X_\nr$ that minimizes the approximate
\begin{align*}
  \sum_{i=1}^{M}w_i\norm{\chi_i-P_V\chi_i}_\X^2,
\end{align*}
of (\ref{eqn:svd_opt}) over all subspaces $V$ of $\X^\nd$ of dimension at most $\nr$, is given by $\X_\nr=\Span{\xi_1,\ldots,\xi_{\nr\wedge K}}$, where $(\lambda_i,\xi_i)_i$ now is the eigenpair sequence of $\C_d$ in which the eigenvalues are ordered in a decreasing fashion and $K$ is the number of positive eigenvalues, see e.g. \cite[Section~1.2]{venturi16}.
\newline
\newline
In a numerical implementation of the weighted POD, the operator $\mathcal{C}_d$ is expressed in a basis. Two possibilities for this are
\begin{itemize}
  \item expressing the operator $\mathcal{C}_d$ in the snapshot basis $(\chi_i)_{i=1}^M$. Let us denote the resulting matrix by $C$, so %\footnote{While $C$ is not symmetric in the standard inner product on $\mathbb{R}^M$, it \textit{is} symmetric in the inner product induced by $C_0$, and thus diagonalizable in that inner product.}
    $C_{ij}=w_i\langle \chi_j,\chi_i\rangle_\X$. Define also the positive definite, symmetric matrix $C_0=(\frac{1}{M}\langle \chi_j,\chi_i\rangle_\X)_{ij}$ and the weight matrix $P=\text{diag}(w_1,\ldots,w_M)$. To get the basis $(\xi_i)_{i=1}^\nr$ of the minimizing subspace $\X_\nr$, the $N$ leading orthonormalized eigenvectors $(\bm{\xi}_i)_{i=1}^\nr$ of $C=MPC_0$ are computed\footnote{Instead of solving $MPC_0x=\lambda x$, it is preferable to solve $MC_0x = \lambda P^{-1}x$, as \gcok{$P^{-1}$ can be easily computed and more efficient solvers are available when $P$ is positive definite\gcok{.}}}. As these are expressed in the snapshot basis, we expand each of them in this snapshot basis to obtain $(\xi_i)_{i=1}^\nr$. That is, the $i$th basis function is $\xi_i=\sum_{j=1}^M(\bm{\xi}_i)_j\chi_j$.
  \item Another possibility, when the weights are all positive, is to express $\mathcal{C}_d$ in the weighted snapshot basis $(\sqrt{w_i}\chi_i)_{i=1}^M$. We denote the resulting matrix as $C_w$, so $(C_w)_{ij} = \sqrt{w_i}\sqrt{w_j}\langle \chi_j,\chi_i\rangle_\X$. Notice that $C_w$ is positive definite and symmetric. The $\nr$ leading orthonormalized eigenvectors $(\bm{\xi}^w_i)_{i=1}^\nr$ of $C_w$ are computed. As they are expressed in the weighted snapshot basis, the functions $(\xi_i)_{i=1}^\nr$ can be recovered by expanding $(\bm{\xi}^w_i)_{i=1}^\nr$ in this weighted snapshot basis: $\xi_i = \sum_{j=1}^{M}(\bm{\xi}^w_i)_j\sqrt{w_j}\chi_j$.
\end{itemize}
The weighted POD thus picks out the most important directions\gcok{, according to an $L^2$-type criterion,} of $\Span{\chi_1,\ldots,\chi_M}$. As only one solution manifold is involved, this is known as the \textit{monolithic approach}. Instead, we shall perform separate weighted PODs on each of the solution manifolds
\begin{align*}
  \{y^\nd(\mu_1),\ldots,y^\nd(\mu_M)\},\hspace{5mm}\{u^\nd(\mu_1),\ldots,u^\nd(\mu_M)\},\hspace{5mm}\{p^\nd(\mu_1),\ldots,p^\nd(\mu_M)\}.
\end{align*}
The resulting low dimensional spaces $Y_\nr,\ U_\nr$ and $P_\nr$ are combined to furnish the reduced basis space $\X_\nr=Y_\nr\times U_\nr\times P_\nr$. This approach, known as the \textit{partitioned approach}, has been observed to be preferable to the monolithic approach in a variety of scenarios, by \cite{maria17}. Of course, if a solution manifold is already contained in a subspace of dimension at most $\nr$, say the solution manifold of the control $\left\{ u^\nd(\mu_1),\ldots,u^\nd(\mu_M) \right\}$, no POD is needed and one simply takes the linear span of this manifold as $U_\nr$. If $Y,U$, or $P$ is a product of Hilbert spaces itself, more POD compressions could also be performed, for example one for each space in this product.

\begin{remark}
The term ``POD'' is often reserved for a weighted POD in which a Monte Carlo sampler is used as quadrature rule, \gcok{see \cite{venturi19,venturitorlo19}}.
\end{remark}

\begin{remark}
  Some quadrature rules, such as Gaussian quadrature, rely on appropriate smoothness of the map $\mu\mapsto \chi^\nd(\mu)$.
\end{remark}

\subsection{Online phase}
The online phase now consists of solving the system (\ref{eqn:lin_qdr_opt}) with $Y,U,P$ replaced by $Y_\nr,U_\nr,P_\nr$, respectively, for any $\mu\in\m$ of interest. That is, we perform a Galerkin projection of $(y(\mu),u(\mu),p(\mu))$ onto $Y_\nr\times U_\nr\times P_\nr$. This amounts to solving a system of size $3\nr\times 3\nr$, if $\dim{Y_\nr}=\dim{U_\nr}=\dim{P_\nr}=\nr$. It should be noted that for this reduced problem to be well-posed for ($\P^{\bm{\mu}}$-almost) every $\mu\in\m$, the operator $A_\nr(\mu)\in\B(Y_\nr,(P_\nr)^*)$ defined by $A_\nr(\mu)\tilde{y}_\nr = (A(\mu)\tilde{y})_{|Y_\nr}$ for $\tilde{y}_\nr\in Y_\nr$, should be boundedly invertible for ($\P^{\bm{\mu}}$-almost) every $\mu\in\m$.

\begin{remark}
  It is customary to define an \textit{aggregate space} or \textit{integrated space} $Z_\nr=\Span{Y_\nr,P_\nr}$ and put $\X_\nr=Z_\nr\times U_\nr\times Z_\nr$ as reduced basis space instead. In that case the subspaces of $Y$ and $P$ coincide but the system to solve online is of the larger size $5\nr\times 5\nr$. \gcok{We stress that it is not fundamental to perform this step in order for the OCP to be well-posed. It \textit{is} fundamental that $A_\nr(\mu)$ is boundedly invertible for every $\mu$, as follows from the same argumentation of Section \ref{sec:truth_approximation}. If $A(\mu)$ happens to be coercive, then one can use aggregate spaces to inherit the coercivity and hence invertibility of $A_\nr(\mu)$ from $A(\mu)$. However, in a noncoercive setting, the invertibility of $A_\nr(\mu)$ should be guaranteed in a different way, e.g. by using supremizer solutions, see \cite{huynh13,rozza07}.}
\end{remark}

In a numerical implementation of the online phase, the optimality system must be expressed in a basis of $\X_\nr$. If the computational complexity of this operation is independent of $\nd$, then so is the complexity of the online phase. The problem is that such a basis is itself expressed in the $\nd$-dimensional snapshot basis, so that this can not be guaranteed without additional assumptions. To this end, we require that $A(\bm{\mu}), B(\bm{\mu}), M(\bm{\mu}), \Q(\bm{\mu}), g(\bm{\mu})$ and $z_d(\bm{\mu})$ all adhere to the following separation of the variables.
\begin{assumption}
  We assume the following \textit{affine decompositions} hold (for $\P$-almost every $\mu$):
\begin{alignat*}{2}
  &A(\mu) = \sum_{q=1}^{Q_A}\Lambda^q_A(\mu)A_q,
  \hspace{15mm}&&B(\mu) = \sum_{q=1}^{Q_B}\Lambda^q_B(\mu)B_q,\\
  &M(\mu) = \sum_{q=1}^{Q_M}\Lambda^q_M(\mu)M_q,
  &&\Q(\mu) = \sum_{q=1}^{Q_\Q}\Lambda^q_\Q(\mu)\Q_q,\\
  &g(\mu) = \sum_{q=1}^{Q_g}\Lambda^q_g(\mu)g_q,
  &&z_d(\mu) = \sum_{q=1}^{Q_{z_d}}\Lambda^q_{z_d}(\mu)(z_d)_q,
\end{alignat*}
where
\begin{itemize}
  \item $A_q\in\B(Y,P^*),\ B_q\in\B(U,P^*),\ M_q\in \B(Z,Z^*),\ \Q_q\in\B(U,U^*),\ g_q\in P^*,\ z_d^q\in Z$,
  \item $\Lambda^q_A, \Lambda^q_B, \Lambda^q_M, \Lambda^q_\Q, \Lambda^q_g, \Lambda^q_{z_d}:\m\rightarrow\R$,
  \item $Q_A,Q_B,Q_M,Q_\Q,Q_g,Q_{z_d}\in\N.$
\end{itemize}
  \label{ass:affine}
\end{assumption}
Further argumentation can be found in \cite[Section~3.3]{rozza2016}.
If the above parametric maps are continuous, then under some additional assumptions on the involved operators the solution maps
\begin{align*}
  &\mu\mapsto (y(\mu),u(\mu),p(\mu)),\hspace{5mm}\\&\mu\mapsto (y^\nd(\mu),u^\nd(\mu),p^\nd(\mu)),\hspace{5mm}\\&\mu\mapsto (y_\nr(\mu),u_\nr(\mu),p_\nr(\mu)),
\end{align*}
are continuous as well, and are therefore measurable, bounded, and in $L^2(\m;Y\times U\times P)$, see Proposition 2.2.6 in  \cite{carere20}.

\section{Numerical Applications}
\label{sec:applications}
\graphicspath{{images/golfo/}}
In this section we extend the three environmental applications, initially introduced in \gcok{\cite{mariatesi, maria17}}, by modelling the random variable $\bm{\mu}$ to have a distribution other than the uniform distribution, and assess the results of the ROMs constructed by a weighted POD approach. In these examples, the choice of training and testing set sizes are based on \cite{maria17}. In the first application the spill of hypothetical pollutant in a fluid described by an elliptic PDE is studied. The second and third examples consist of an \gcok{o}cean circulation model, which is based on the quasi-geostrophic equations which form a noncoercive PDE. In these latter two numerical examples, a comparison between the results obtained with and without the usual aggregation procedure is given, but we shall not be too concerned with the well-posedness of the OCPs in these two examples.

\subsection{Hypothetical pollution in the Gulf of Trieste}
We model a hypothetical spill of a pollutant near the harbor of Koper, Slovenia, in the Gulf of Trieste, Italy. It is governed by an elliptic \gcok{steady} advection-diffusion equation.

\paragraph{Problem Formulation}
The domain $\Omega\subset\R^2$ with Lipschitz boundary is an approximation of the geographical area of the Gulf of Trieste. A fine triangulation $\T_h$ of mesh size $h>0$ has been constructed and shared by the authors of \cite{maria17}. In Figure \ref{fig:golfo_domain} (\textit{left/center}) this domain is shown, as well as $\T_h$. Two subdomains of $\Omega$ are of particular importance, the subdomain of the spill $\Omega_u$, and the domain of observation, $\Omega_{obs}$.

\begin{description}
  \item[$\Omega_u$] \hspace{2.6mm}the area in which the pollutant is spilled, corresponding to the geographical area of the harbor of Koper, Slovenia,
  \item[$\Omega_{obs}$] the domain of observation, the Miramare natural reserve, which is of interest being a protected environment due to its ecological flora and fauna and being a prominent area to relax for tourists and the citizens of Trieste.
\end{description}

The boundary $\Gamma:=\partial\Omega$ is subdivided in a Dirichlet boundary $\Gamma_D$ and a Neumann boundary $\Gamma_N$, where $\Gamma=\Gamma_D\cup\Gamma_N$ and $\Gamma_D\cap\Gamma_N=\emptyset$, and on which homogeneous Dirichlet and Neumann boundary conditions are imposed, respectively. The Dirichlet boundary corresponds to the coastal part of $\Gamma$ while the Neumann boundary corresponds to open sea, see Figure \ref{fig:golfo_domain}.

\begin{figure}[!htb]
  \makebox[\textwidth][c]{
    \hspace{12mm}
    \begin{subfigure}{0.5\textwidth}
      \vspace{-52.5mm}
      \includegraphics[width=0.6\textwidth]{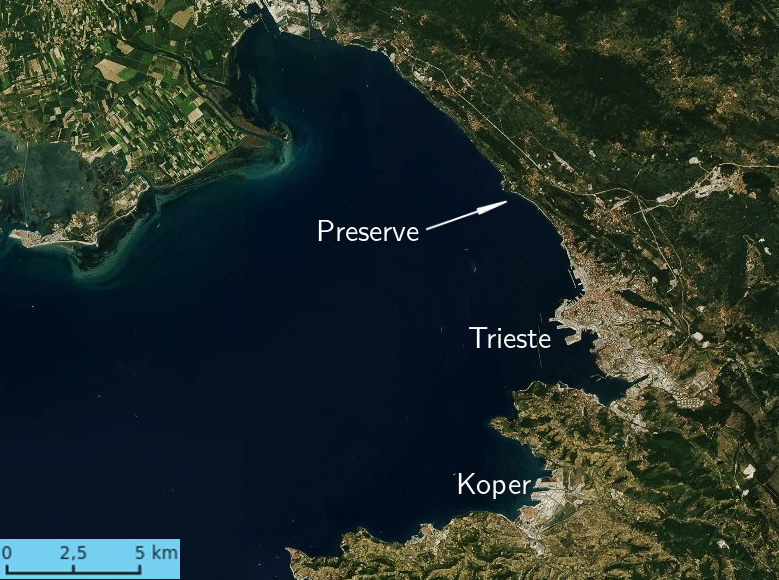}
    \end{subfigure}
    \hspace{-52mm}
    \begin{subfigure}[b]{0.6\textwidth}
      \includegraphics[width=1.1\textwidth]{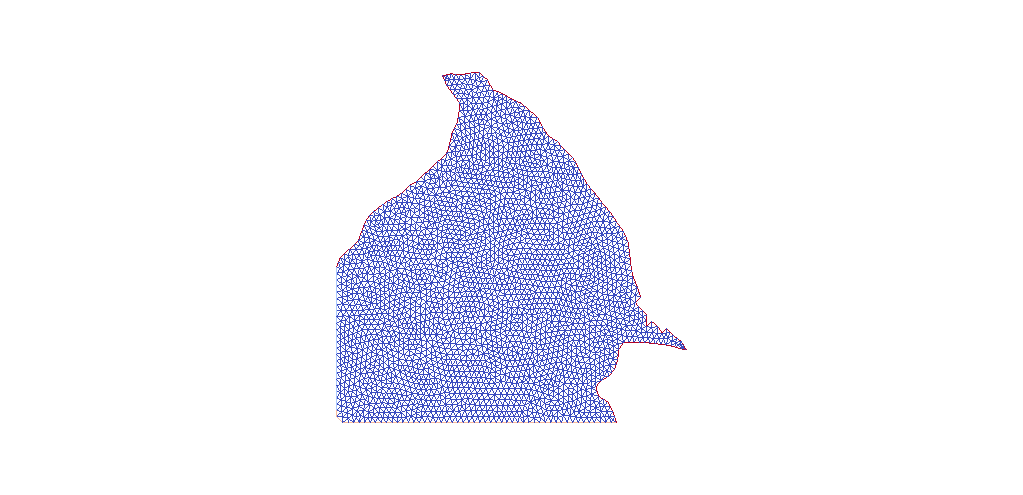}
    \end{subfigure}
    \hspace{-45mm}
    \begin{subfigure}[b]{0.6\textwidth}
      \includegraphics[width=1.1\textwidth]{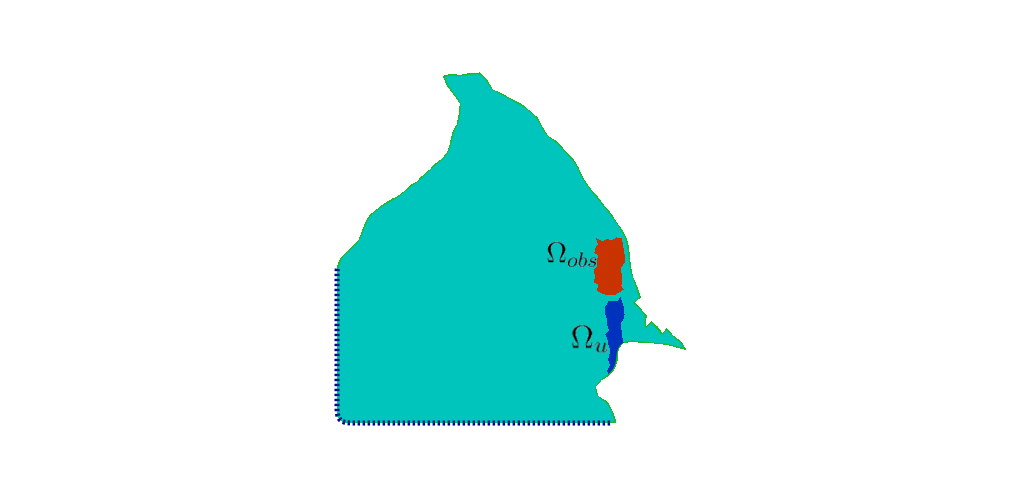}
    \end{subfigure}
  }
  \caption{\footnotesize{\textit{left}: geographical area of the Gulf of Trieste, \textit{center}: physical domain $\Omega$ with triangulation, \textit{right}: physical domain with subdomains $\Omega_u$ and $\Omega_{obs}$, and sub boundaries $\Gamma_N$ (dotted blue) and $\Gamma_D$ (solid green).}}
  \label{fig:golfo_domain}
\end{figure}

Let us denote the state variable $y$ as the pollutant concentration, while the desired concentration $z_d=0.2\chi_{\Omega_{obs}}\in L^2(\Omega)$ represents the maximal safe concentration of pollutant on $\Omega_{obs}$, where $\chi_{\Omega_{obs}}$ is the indicator of $\Omega_{obs}$. The natural state space is $Y:=H^1_{\Gamma_D}(\Omega)=\left\{ v\in H^1(\Omega):\ v=0\text{ on }\Gamma_D \right\}$. For $u\in U:=\R$, consider $u_0=u\chi_{\Omega_u}\in L^2(\Omega)$. The function $u_0$ models a source on $\Omega_u$ originating from the spill. %The larger the pollutant concentration $u_0$, the larger that the value of $y$ gets.
We are interested in finding the value of $u$ for which the pollutant concentration in $\Omega_{obs}$ equals the maximal safe concentration $z_d$, or is as close to $z_d$ as possible in an $L^2(\Omega)$ sense.
\newline
\newline
The problem above can be phrased as the following OCP, where the governing advection-diffusion state equation is written directly in its weak form:

\begin{align*}
  &\text{minimize } J(y,u) = \frac{1}{2}\int_{\Omega_{obs}} |y-z_d|^2 \d{x} + \frac{\alpha}{2}\int_{\Omega_u}|u|^2\d{x}\text{ over all }(y,u)\in Y\times U,\\
  &\text{such that } A(\mu)y+B(\mu)u = 0,\text{ where}\\
  & \langle A(\mu)y,\tilde{p}\rangle_{P^*P} = \int_\Omega\mu_1\grad{y}\cdot\grad{\tilde{p}}\d{x} + \int_{\Omega} \left([\mu_2,\mu_3]\cdot\grad{y}\right)\tilde{p}\d{x},\\
  & \langle B(\mu) u,\tilde{p}\rangle_{P^*P} = -L_0u\int_{\Omega_u}\tilde{p}\d{x},\\
  & \text{and } P=Y.
\end{align*}

We take $\m=[\frac{1}{2},1]\times[-1,1]\times[-1,1].$ The parameter $\mu$ describes specific properties of the sea: $\mu_1$ is a diffusivity parameter while $[\mu_2,\mu_3]$ models a constant advective field.
The constant $L_0=1000$ is used to put the state equation in nondimensional form. We are mostly interested in minimizing the first term in the objective functional. On the other hand, to ensure uniqueness of the optimal control, we do prescribe a small positive value for $\alpha$, namely $\alpha=10^{-7}$. When $A(\mu)$ is coercive for every $\mu$, it is in particular invertible for every $\mu\in\m$ so that the OCP is well-posed for every $\mu$ by Corollary \ref{cor:general_linear_quadratic}, and to find the solution we must solve (\ref{eqn:lin_qdr_opt}).
Notice that in the formulation of Definition \ref{def:linear_quadratic}, the observation space $Z$ is taken to be $L^2(\Omega)$, $M$ and $\Q/\alpha$ are the Riesz map on $L^2(\Omega)$, $C$ is the injection $Y\hookrightarrow Z$ and $g=0$. 

\paragraph{Coercivity of $A_\nr(\mu)$} Coercivity of $A_\nr(\mu)$ holds if the Poincar\'e and Trace constants $C_p$ and $C_t$, given in the inequalities
\begin{align}
\label{eqn:poincare}
  C_p &=  \inf\left\{ c>0:\ \int_{\Omega}^{}|u|^2\d{x}\leq c\int_{\Omega}^{}|\grad{u}|^2\d{x}\hspace{5mm} \forall u\in H^1_{\Gamma_D}(\Omega) \right\},\\
  C_t &=  \inf\left\{ c>0:\ \int_{\Gamma_N}^{}|u|^2\d{x} \leq c\int_{\Omega}^{}|u|^2+|\grad{u}|^2\d{x}\hspace{5mm} \forall u\in H^1_{\Gamma_D}(\Omega) \right\},  \nonumber
\end{align}
are small enough: $(C_p+1)C_t < \frac{1}{2}\sqrt{2}$. In order to see this, notice that $\Gamma_N$ consists of a western vertical part $W$, a southern horizontal part $S$, and a small diagonal part $SW$ angled at $45$ degrees that joins $W$ and $S$. With a convective field $[\mu_2,\mu_3]$ and outer normal $n$ at the boundary, we have on $\Gamma_N$ that
\[
  [\mu_2,\mu_3]\cdot n = \mu_2\chi_{W}+\mu_3\chi_{S}+\frac{1}{2}\sqrt{2}(\mu_2+\mu_3)\chi_{SW}\geq -\chi_{W}-\chi_{E}-\sqrt{2}\chi_{SW}\geq-\sqrt{2}\chi_{\Gamma_N},
\]
Writing $y\grad{y} = \frac{1}{2}\grad{y}^2$ and using Green's Theorem and the Trace and Poincar\'e inequalities, we obtain for arbitrary $(\mu_1,\mu_2,\mu_3)\in\m$ and $y\in H^1_{\Gamma_D}(\Omega)$
\begin{align*}
  \langle A(\mu)y,y\rangle_{Y^*Y} &= \mu_1 \int_{\Omega}^{}|\grad{y}|^2\d{x} + \frac{1}{2}\int_{\Gamma_N}^{}([\mu_2,\mu_3]\cdot n)y^2\d{x},\\
  &\geq \frac{1}{2}\int_{\Omega}^{}|\grad{y}|^2\d{x} - \frac{1}{2}\sqrt{2}\int_{\Gamma_N}^{}|y|^2\d{x}\\
  &\geq \frac{1}{2}\int_{\Omega}^{}|\grad{y}|^2\d{x} - \frac{1}{2}\sqrt{2}C_t\int_{\Omega}^{}(|y|^2+|\grad{y}|^2)\d{x}\\
  &\geq \frac{1}{2}\int_{\Omega}^{}|\grad{y}|^2\d{x}\left( 1-\sqrt{2}C_t-\sqrt{2}C_tC_p \right).
\end{align*}
Since by again the Poincar\'e inequality (\ref{eqn:poincare}) the $H^1(\Omega)$-seminorm $y\mapsto \int_{\Omega}^{}|\grad{y}|^2\d{x}$ is equivalent to the $H^1(\Omega)$-norm on $H^1_{\Gamma_D}(\Omega)$, we see that coercivity is ensured if $(C_p+1)C_t < \frac{1}{2}\sqrt{2}$.

\paragraph{Truth Approximation}
The high fidelity spaces are based on the Finite Elements
\begin{align}
  X^k_h=\left\{ v\in C^0(\Omega):\ v_{|K}\in P_k(K)\ \forall K\in\T_h \right\} \subset H^1(\Omega),
  \label{eqn:finite_elements}
\end{align}
where $P_k(K)$ is the space of polynomials on the element $K$ of degree at most $k$.
More precisely, as the high fidelity state, control, and adjoint spaces $Y^\nd,U^\nd,P^\nd$ we take
\[
  Y^{\nd_Y}=X^1_h\cap Y,\hspace{5mm} U^{\nd_U} = \R,\hspace{5mm} P^{\nd_P} = X^1_h\cap Y.
\]
The trace and Poincar\'e inequalities remain valid on $Y^\nd$, with the (smaller) corresponding constants $C^\nd_t$ and $C^\nd_p$. By solving an eigenvalue problem we computed these constants to be $C^\nd_t=0.52$ and $C^\nd_p=0.06$. Hence, coercivity of $A^\nd(\mu)$ is guaranteed for every $\mu\in\m$ since we take $Y^\nd=P^\nd$. The truth problem described in Section \ref{sec:truth_approximation} thus is well-posed by Corollary \ref{cor:general_linear_quadratic}. Furthermore, the truth solution can be shown to converge to the \gcok{solution of the problem in continuous formulation} uniformly on $\m$ (see \cite{carere20}).

\paragraph{Reduced Order Model}
In this case we need only perform a POD compression on state and adjoint, and can leave $U_\nr=U$. We do aggregate state and adjoint in this application, so that $A_\nr(\mu)$ inherits the coercivity of $A(\mu)$ for every $\mu$. If $\dim{Y_\nr}=\dim{P_\nr}=\nr$, then it remains to solve a system of size $(4\nr +1)\times(4\nr+1)$, which now is well-posed due to the inherited coercivity for $A_\nr(\mu)$ for each $\mu$. Because $\Q,M,B,g$ and $z_d$ are parameter independent and $A$ can be decomposed in the $Q_A=3$ terms
\begin{alignat*}{2}
  &\Lambda^1_A(\mu) = \mu_1,\hspace{10mm} && \langle A_1y,\tilde{y}\rangle_{Y^*Y} = \int_\Omega \grad{y}\cdot\grad{\tilde{y}}\d{x},\\
  &\Lambda^2_A(\mu) = \mu_2, && \langle A_2y,\tilde{y}\rangle_{Y^*Y}=\int_{\Omega}^{}\frac{\partial y}{\partial x_1}\tilde{y}\d{x},\\
  &\Lambda^3_A(\mu) = \mu_3, && \langle A_3y,\tilde{y}\rangle_{Y^*Y}=\int_{\Omega}^{}\frac{\partial y}{\partial x_2}\tilde{y}\d{x}\reviewerB{.}
\end{alignat*}
Assumption \ref{ass:affine} holds, hence this system can be solved in a computation count independent of $\nd$.

\paragraph{Training set generation}
As mentioned in Section \ref{sec:rom}, the weighted POD is based on the choice of a quadrature rule with nodes $\m_d=\{\mu_1,\ldots\mu_M\}$ and weights $\{w_1,\ldots,w_M\}$. Writing $\bm{\mu} = (\bm{\mu}_1,\bm{\mu}_2,\bm{\mu}_3)$, we shall only consider product measures $\P^{\bm{\mu}}=\P^{\bm{\mu}_1}\times \P^{\bm{\mu}_2}\times\P^{\bm{\mu}_3}$ for which each $\P^{\bm{\mu}_i}$ admits a Lebesgue density, and we will assess the performance of the following quadrature rules (see e.g. \cite{sulli15}):
\begin{itemize}
  \item a Monte Carlo sampler, which samples the nodes from $\P^{\bm{\mu}}$. Its weights are all equal to $\frac{1}{M}$.
  \item The tensor product of three Gaussian quadrature rules. In this case the weights are not all equal.
  \item A Pseudo-Random sampler, that provides a rule for Un$([0,1]\times [0,1]\times [0,1])$, the Uniform distribution on $[0,1]\times [0,1]\times [0,1]$. To get a rule for $\P^{\bm{\mu}}$ on $\m$, we use the method of inversion for each $\P^{\bm{\mu}_i}$.
  \item A tensor product of Clenshaw-Curtis quadrature rules, which provide\footnote{The Clenshaw-Curtis quadrature rule can also be implemented to provide a rule for different densities.} a rule for Un$([-1,1])$. To get a rule for $\P^{\bm{\mu}}$, we use a change of variables for each $i$.
\end{itemize}
\paragraph{Results}
An image of the truth solution for the state and adjoint component is shown in Figure \ref{fig:golfo_offline}. The parameter value for which this solution is obtained is $\mu=(1,-1,1)$. This solution is thus obtained with a convective field $(\mu_2, \mu_3) =(-1,1)$ which models a water flow from south-east to north-west. The value of  the optimal control is $7.4\times 10^{-1}$.
\begin{figure}[!htb]
  \makebox[\textwidth][c]{
      \hspace{-27mm}
      \begin{subfigure}[b]{0.55\textwidth}
        \includegraphics[width=1.3\textwidth]{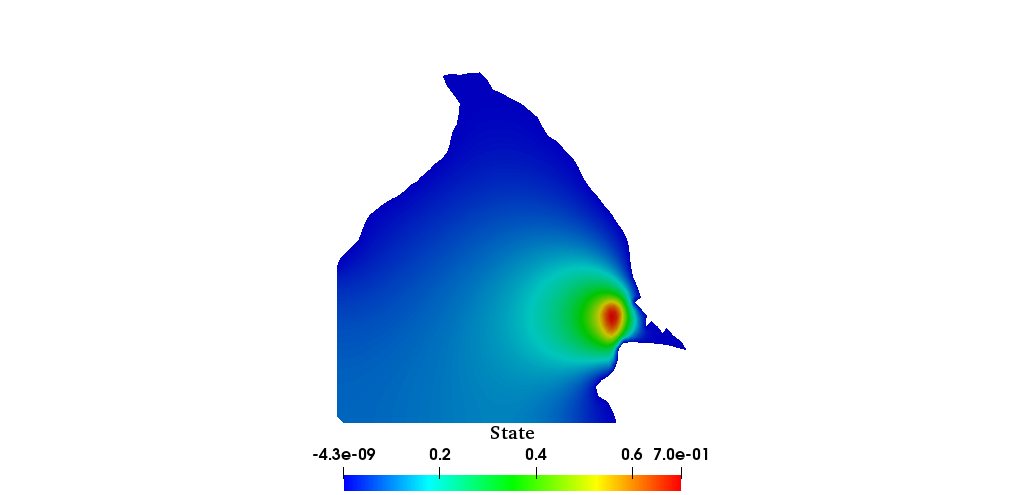}
      \end{subfigure}
      \hspace{-40mm}
      \begin{subfigure}[b]{0.55\textwidth}
        \includegraphics[width=1.3\textwidth]{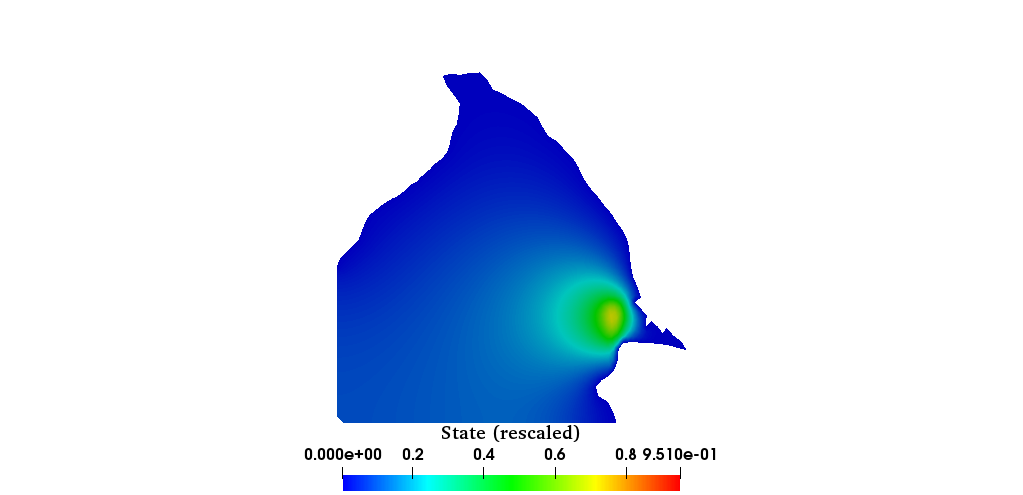}
      \end{subfigure}
      \hspace{-40mm}
      \begin{subfigure}[b]{0.55\textwidth}
        \includegraphics[width=1.3\textwidth]{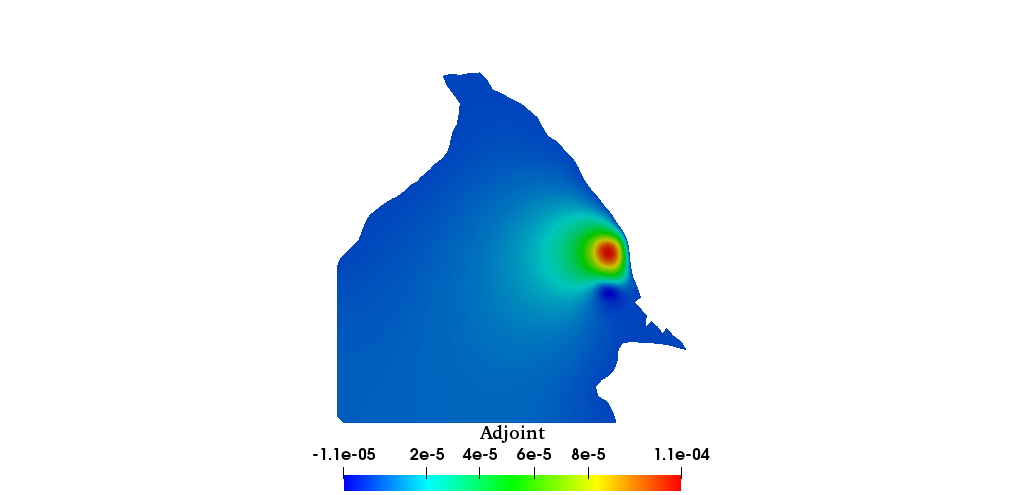}
      \end{subfigure}
  }
  \caption{\footnotesize{Gulf: truth solution for state (\textit{left}, \textit{center}) and adjoint (\textit{right}), for $\mu=(1,-1,1)$.}}
  \label{fig:golfo_offline}
\end{figure}
Notice that the truth solution for the state is strictly smaller than $z_d$ on $\Omega_{obs}$.
To enable visual comparison with the result of \cite{maria17}, the state solution is also displayed with a rescaled color range\reviewerB{.}
The dimension of the truth problem is\footnote{To be precise, we are reporting the total number of nonzero coefficients of the expansion of the truth solution components $y^\nd, p^\nd$ for $\mu=(1,-1,1)$ in the Finite Elements basis (\ref{eqn:finite_elements}).} $5345\times 5345$.
\newline
\newline
Supposing first that $\bm{\mu}$ is uniformly distributed, we build ROMs with $\nr=1,\ldots,35$, which are constructed with a Monte Carlo sampling procedure to obtain a training set $\m_d$ of size 100.
For $\mu=(1,-1,1)$, a plot of the pointwise difference between the reduced state and adjoint solutions obtained for $\nr=35$ and the respective truth solutions is shown in Figure \ref{fig:golfo_diff}.% The value of the objective functional for this parameter value approximates $\J^\nd$ exactly up to 12 decimal places.
\newline
Next, we consider the relative error for state, control and adjoint solution components and the output functional, i.e.
\begin{align*}
  e_{y,\nr}(\mu)&=\frac{{\norm{y^\nd(\mu)-y_\nr(\mu)}_Y}}{\norm{y^\nd(\mu)}_Y},\\ e_{u,\nr}(\mu)&=\frac{{\norm{u^\nd(\mu)-u_\nr(\mu)}_U}}{\norm{u^\nd(\mu)}_U},\\ e_{p,\nr}(\mu)&=\frac{\norm{p^\nd(\mu)-p_\nr(\mu)}_P}{\norm{p^\nd(\mu)}_P},\\ e_{J,\nr}(\mu)&=\frac{{|J^\nd(\mu)-J_\nr(\mu)|}}{|J^\nd(\mu)|}.
\end{align*}
We plot their base-10 logarithm, averaged over values of $\mu$ in a testing set of size 100 which is sampled from $\P^{\bm{\mu}}$ \reviewerA{and is different from the training set}, in Figure \ref{fig:golfo_errors_uniform_mc}. The sample \gcok{average} gives an indication of the trend of the decay. Not all parameters follow this trend exactly. For parameter value for which the problem is inherently more difficult to solve, the error decays more slowly. The amount of variation of the logarithmic \reviewerA{relative} errors among the parameters in the training set is indicated in the plots, by including the sample standard deviation. For example, for the state this is an unbiased estimator of the true standard deviation
\[
  \sqrt{\int_{\Theta}^{}\left( \log_{10}\norm{e_{y,\nr}(\bm{\mu})}_Y-\mathbb{E}\log_{10}\norm{e_{y,\nr}(\bm{\mu})}_Y \right)^2\d\P}.
\]
The same holds true for the control and adjoint components and the output functional.

\begin{figure}[!htb]
  \makebox[\textwidth][c]{
      \hspace{-27mm}
      \begin{subfigure}[b]{0.55\textwidth}
        \includegraphics[width=1.3\textwidth]{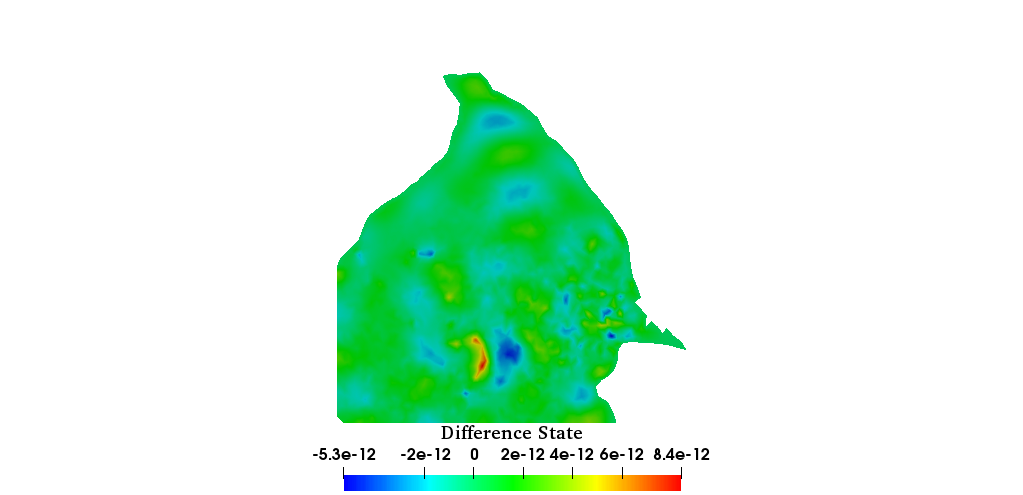}
      \end{subfigure}
      \hspace{-40mm}
      \begin{subfigure}[b]{0.55\textwidth}
        \includegraphics[width=1.3\textwidth]{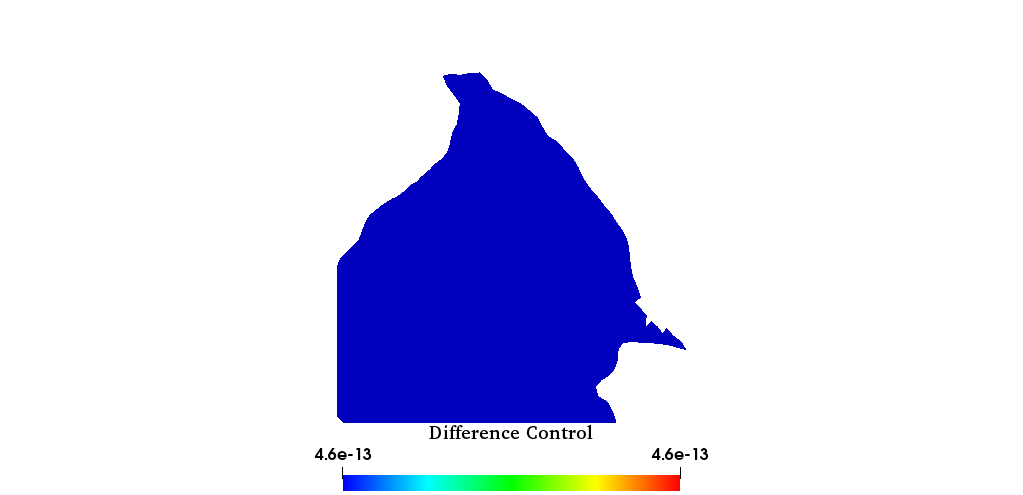}
      \end{subfigure}
      \hspace{-40mm}
      \begin{subfigure}[b]{0.55\textwidth}
        \includegraphics[width=1.3\textwidth]{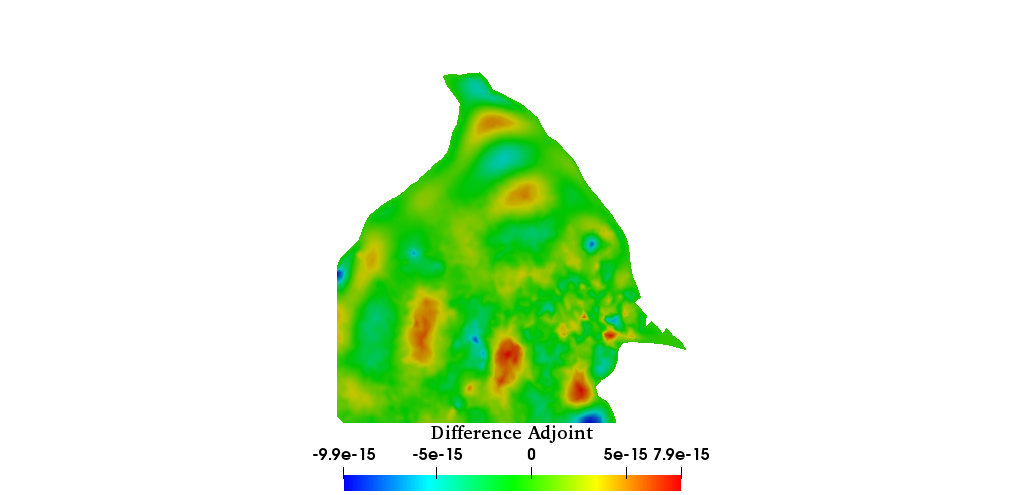}
      \end{subfigure}
  }
  \caption{\footnotesize{Gulf: ROM obtained with Monte Carlo sampling from uniform distribution: pointwise difference between reduced and truth state, control and adjoint solutions for $\mu=(1,-1,1)$ (from left to right).}}
  \label{fig:golfo_diff}
\end{figure}
\begin{figure}[!htb]
  \makebox[\textwidth][c]{
    \hspace{-2mm}
    \begin{subfigure}[b]{0.4\textwidth}
      \includegraphics[width=1.0\textwidth]{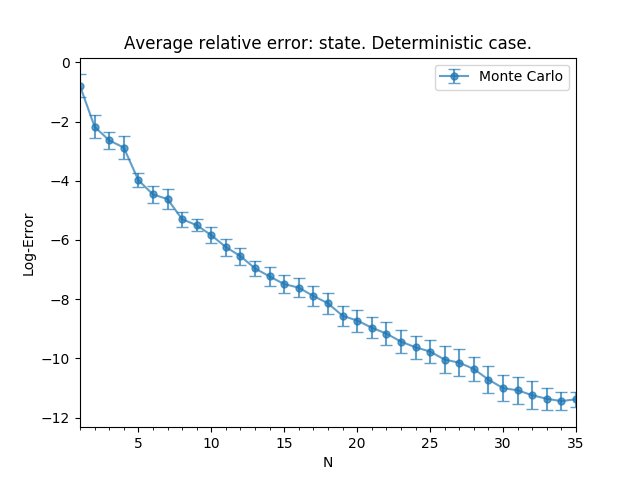}
    \end{subfigure}
    %\hspace{-28mm}
    \begin{subfigure}[b]{0.4\textwidth}
      \includegraphics[width=1.0\textwidth]{relative_error_u_Uniform_mc.png}
    \end{subfigure}
  }
  \makebox[\textwidth][c]{
    \hspace{-2mm}
    \begin{subfigure}[b]{0.4\textwidth}
      \includegraphics[width=1.0\textwidth]{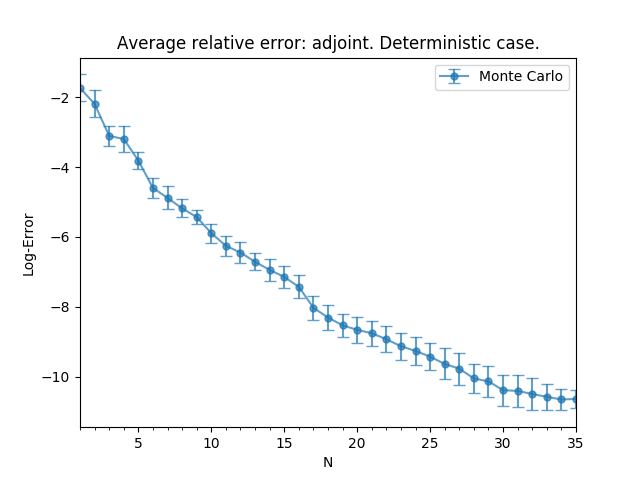}
    \end{subfigure}
    %\hspace{-28mm}
    \begin{subfigure}[b]{0.4\textwidth}
      \includegraphics[width=1.0\textwidth]{relative_error_output_Uniform_mc.png}
    \end{subfigure}
  }
  \caption{\footnotesize{Gulf: average logarithmic relative errors for state (\textit{top left}), control (\textit{top right}), adjoint (\textit{bottom left}) and output (\textit{bottom right}) \reviewerA{in the deterministic case}, as a function of $\nr$. \reviewerA{ROM obtained via Monte Carlo sampling.}}}
  \label{fig:golfo_errors_uniform_mc}
\end{figure}
The pointwise errors for state, control and adjoint are very small and the relative normed errors decay exponentially. The ROM corresponding to $\nr=35$ reaches a high accuracy of order $10^{-11}$ for the state variable\footnote{After $\nr=35$, the error typically increases again due to numerical error, originating from the normalization procedure after the weighted POD, which is large for eigenvectors corresponding to the very small eigenvalues. We shall only plot the phase of decay.}.
\newline
The different quadrature rules specified above are implemented as well. In Figure \ref{fig:golfo_errors_uniform_rules} (\textit{left}) the state error originating from ROMs that are constructed with Clenshaw-Curtis and Gaussian quadrature rules as specified above, are compared with the Monte Carlo sampler. A comparison is also made with a Pseudo Random sampler. The size of the training sets for the Monte Carlo and Pseudo Random samplers is 100. \reviewerB{For each} of the three parameters, a Gaussian and Clenshaw-Curtis quadrature of five nodes is taken, leading to tensor product rules with a training set size of 125. All four quadrature rules perform similarly.
Finally, $\bm{\mu}$ is given a Beta(75,75) distribution for all three parameters, which puts most probability mass around the center of the parameter domain. The ROM is effectively using the additional information, as it requires only half the number of reduced functions compared to the deterministic case. This can be seen in Figure \ref{fig:golfo_errors_uniform_rules} (\textit{right}). The  ROMs, constructed with the Pseudo Random, Gauss and Monte Carlo rule, all perform similarly.
\begin{figure}[!htb]
  \makebox[\textwidth][c]{
    \hspace{2mm}
    \begin{subfigure}[b]{0.4\textwidth}
      \includegraphics[width=1.0\textwidth]{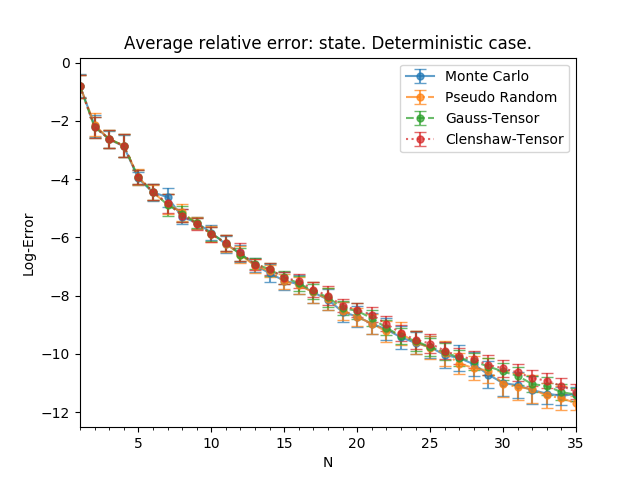}
    \end{subfigure}
    %\hspace{-28mm}
    \begin{subfigure}[b]{0.4\textwidth}
      \includegraphics[width=1.0\textwidth]{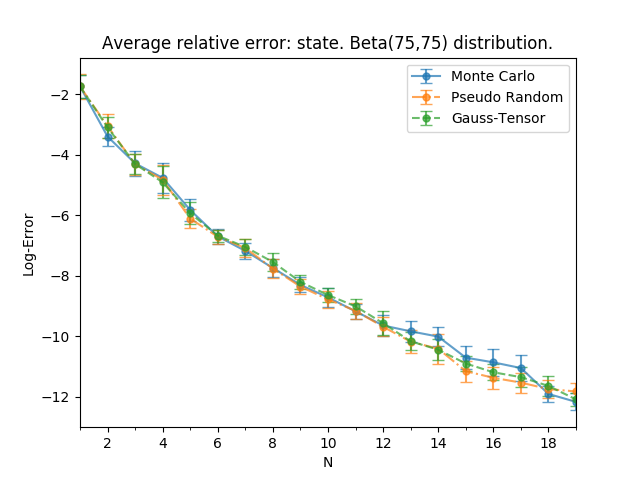}
    \end{subfigure}
    %\hspace{-28mm}
    %\begin{subfigure}[b]{0.5\textwidth}
    %  \includegraphics[width=1.0\textwidth]{relative_error_y_Uniform_rules_smolyak.png}
    %\end{subfigure}
  }
  \caption{\footnotesize{Gulf: average logarithmic relative errors for the state in the deterministic case (\textit{left}), and under a Beta(75,75) distribution (\textit{right}), as a function of $\nr$. Various quadrature rules are compared.}}
  \label{fig:golfo_errors_uniform_rules}
\end{figure}
  \begin{table}[!htb]
  \centering
  \begin{tabular}{|r|r|r|r|c|}
    \hline
    $\nr$ & \gcok{average} & min & max & deviation \\
    \hline
    2 & 52.8 & 26.8 & 81.5 & 6.3 \\
    7 & 51.4 & 23.7 & 72.1 & 6.7 \\
    14 & 46.1 & 18.6 & 69.7 & 7.4\\
    21 & 42.6 & 19.1 & 63.3 & 5.8 \\
    28 & 35.9 & 19.2 & 50.5 & 5.1 \\
    35 & 31.2 & 18.5 & 45.7 & 4.1 \\
    \hline
  \end{tabular}
  \caption{\footnotesize{Gulf: sample \gcok{average}, minimal value, maximal value, and sample standard deviation of speedup-index obtained with a testing set of size 100. Deterministic case. ROM obtained by Monte Carlo sampling.}}
  \label{tab:golfo_speedup}
\end{table}
Having thus concluded that the proposed ROM is able to make use of a low dimensional solution space which remains very accurate, we \gcok{further comment on} efficiency. To verify it, for each $\mu$ in the testing set we compute the so-called \textit{speedup-index}:\[\text{speedup-index} =\frac{\text{computation time of truth solution}}{\text{computation time of reduced \gcok{basis} approximation}}.\] Table \ref{tab:golfo_speedup} lists\footnote{The speedup-index is machine dependent. The results are obtained with 6GB of RAM and a 2.60 GHz i5-3230M CPU.}, for a few values of $\nr$, the sample \gcok{average} and sample standard deviation of the computed speedup-indices over the testing set of size 100, in the deterministic case. The minimal and maximal speedup-index is also displayed. For $\nr=35$, the average computational saving is around a factor 30, which comes down to the difference between one month and one day.

\subsection{Linearized quasi-geostrophic equation on the Atlantic Ocean}
\label{subsec:linear_atlantic}
\graphicspath{{images/geostrophic/}}
Let $\Omega\subset\R^2$ be open, bounded and with Lipschitz boundary. The scalar solution fields $\vv,\rho$ on $\Omega$ are said to satisfy the \gcok{steady} one-layer quasi-geostrophic equation in streamline-vorticity formulation if they solve, see \cite{crisciani2013}, the equations
\begin{equation*}
  \begin{cases}
    \begin{aligned}
      \rho &= \lap{\vv} &\text{ in }\Omega,\\
      \mu_3\F(\vv,\rho)+ \frac{\partial\vv}{\partial x_1} + \mu_1\rho - \mu_2\lap{\rho} &= u&\text{ in }\Omega,\\
      \vv &= 0&\text{ on }\partial\Omega,\\
      \rho &= 0&\text{ on }\partial\Omega\reviewerB{,}
    \end{aligned}
  \end{cases}
\end{equation*}
where $\mu=(\mu_1,\mu_2,\mu_3)\in \R_+\times\R_+\times\R_{\geq 0}$ are physical parameters representing fluid properties, $u$ represents wind stress and $\F$ describes a nonlinearity given by 
\[
  \F(\vv,\rho) = \frac{\partial \vv}{\partial x_1}\frac{\partial \rho}{\partial x_2} - \frac{\partial\vv}{\partial x_2}\frac{\partial \rho}{\partial x_1},
\]
for all $\vv,\rho$ in a suitable function space specified later. In practice, the parameters $\mu_1$ and $\mu_3$ are smaller than unity. 
If $\bm{v}$ denotes the velocity field of a fluid, then $\bm{v}$ can be recovered from $\vv$ as $\bm{v}=(-\frac{\partial\vv}{\partial x_2},\frac{\partial\vv}{\partial x_1})$. As parameter space we take $\m=[10^{-4},1]\times[0.07^3,1]\times[10^{-4},0.045^{2}]$.
\newline
\newline
We shall work with the corresponding weak form:
\begin{align*}
  \int_{\Omega}^{}\rho\tilde{\rho}\d{x} + \int_{\Omega}^{}\grad{\vv}\reviewerB{\cdot}\grad{\tilde{\rho}}\d{x} &=  0,\\
  \mu_3\int_{\Omega}^{}\F(\vv,\rho)\tilde{\vv}\d{x}+\int_{\Omega}^{}\frac{\partial \vv}{\partial x_1}\tilde{\vv}\d{x} + \mu_1\int_{\Omega}^{}\rho\tilde{\vv}\d{x} +\mu_2 \int_{\Omega}^{}\grad{\rho}\reviewerB{\cdot}\grad{\tilde{\vv}}\d{x} &=  0,
\end{align*}
where $\vv,\rho,\tilde{\vv},\tilde{\rho}\in H^1_0(\Omega)$. Defining $Y=H^1_0(\Omega)\times H^1_0(\Omega)$, $P=Y$ and $U=L^2(\Omega)$, this can be written as
\begin{align} 
\label{eqn:geostrophic_state}
  A(\mu)y+B(\mu)u = 0,
\end{align}
where, for all $\mu\in\m$, $A(\mu):Y\rightarrow P^*$ and $B(\mu)\in \B(U,P^*)$ are given by
\begin{align*}
  \langle A(\mu) (\vv,\rho),(\tilde{\vv},\tilde{\rho})\rangle_{P^*P} &= \langle A_0(\mu)(\vv,\rho),(\tilde{\vv},\tilde{\rho})\rangle_{P^*P} + \mu_3\int_{\Omega}^{}\F(\vv,\rho)\tilde{\vv}\d{x}, \\
  \langle A_0(\mu) (\vv,\rho),(\tilde{\vv},\tilde{\rho})\rangle_{P^*P} &= \int_\Omega \frac{\partial \vv}{\partial x_1}\tilde{\vv}\d{x} + \mu_2\int_{\Omega}^{}\grad{\rho}\reviewerB{\cdot}\grad{\tilde{\vv}}\d{x} + \mu_1\int_{\Omega}^{}\rho\tilde{\vv}\d{x} \\&+ \int_{\Omega}^{}\rho\tilde{\rho}\d{x} + \int_{\Omega}^{}\grad{\vv}\reviewerB{\cdot}\grad{\tilde{\rho}}\d{x},\\
  \langle B(\mu)u,(\tilde{\vv},\tilde{\rho})\rangle_{P^*P} &= -\int_\Omega u\tilde{\vv}\d{x}.
\end{align*}
Using the Cauchy-Schwarz inequality it is easy to see that $B(\mu)$ is bounded with $\norm{B(\mu)}_{\B(U,P^*)}\leq 1$, and that $A_0(\mu)$ is linear and \gcok{also} bounded with $\norm{A_0(\mu)}_{\B(Y,P^*)} \leq (3+|\mu_1|+|\mu_2|)$.
\begin{remark}
  For the state equation to be almost surely well-posed it is required that $A(\mu)$ is invertible for $\P^{\bm{\mu}}$-almost every $\mu\in\m$. It is shown in \cite{barcilon88} that $A(\mu)$ is invertible if $\mu_1$ is small enough in comparison to $\mu_2$ and if one instead takes $Y=P=\left( H^2(\Omega)\cap H^1_0(\Omega) \right)\times \left(H^2(\Omega)\cap H^1_0(\Omega)\right)$. Furthermore, \cite{barcilon88} also requires $u\in L^\infty(\Omega)$, leading to a nonreflexive control space. Due to the numerical success obtained with (\ref{eqn:geostrophic_state}) (and spaces defined therein) in \cite{maria17}, we shall continue to use (\ref{eqn:geostrophic_state}) in this work. Stabilization procedures could also be used, such as in \cite{kim15}.
\end{remark}

Before considering an OCP governed by the nonlinear quasi-geostrophic equation, we consider a linearized version by taking $\mu_3=0$. Therefore, $A(\mu)=A_0(\mu)\in\B(Y,P^*)$. This version of the quasi-geostrophic equation is also known as the linear Stommel-Munk model, \cite{crisciani2013, kim15}.

\paragraph{Problem formulation}
As physical domain $\Omega$ we take the one from \cite{maria17}: the part of the Atlantic Ocean between the coasts of Florida North-Africa and Southern Europe. The authors have kindly shared a scaled model of this domain, which we indicate with $\Omega$, as well as a fine mesh $\T_h$ on $\Omega$. The mesh size is denote by $h$.
In Figure \ref{fig:atlantic_domain}, the mesh is shown.
\newline
\newline
Let $\vv_d\in Z$ represent observed data, with observation space $Z=L^2(\Omega)\times L^2(\Omega)$. Given fluid properties $\mu$, we are interested in finding the action of the wind $u$ that, according to the quasi-geostrophic model, would have generated $\vv_d$, or at least a state that is close to $\vv_d$ in the $L^2(\Omega)$ sense. This can be done by solving the following minimization problem: 
\begin{align*}
  &\text{minimize }J(\vv,\rho,u) = \frac{1}{2}\int_{\Omega}|\vv-\vv_d|^2\d{x} +\frac{\alpha}{2}\int_\Omega|u|^2\d{x}\\&\text{over all }(\vv,\rho)\in Y,u\in U,\\
  &\text{such that }A(\mu)(\vv,\rho)+B(\mu)u=0.
\end{align*}
The constant $\alpha$ is given a small but nonzero value $10^{-5}$ for the problem to have at most one solution. 
\begin{figure}[!htb]
  \centering
  %\makebox[\textwidth][c]{
  % \hspace{12mm}
  % \begin{subfigure}{0.5\textwidth}
  %   \vvspace{-52.5mm}
      \includegraphics[width=0.7\textwidth]{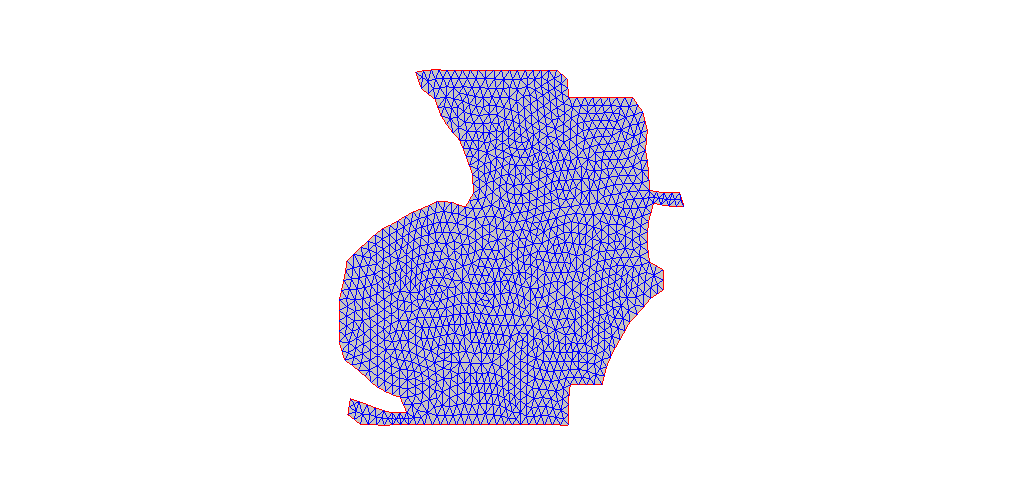}
  % \end{subfigure}
  % \hspace{-52mm}
  % \begin{subfigure}[b]{0.6\textwidth}
  %   \includegraphics[width=1.3\textwidth]{mesh_golfo.png}
  % \end{subfigure}
  % \hspace{-45mm}
  % \begin{subfigure}[b]{0.6\textwidth}
  %   \includegraphics[width=1.3\textwidth]{subdomains_golfo_complete.png}
  % \end{subfigure}
  %}
  \caption{\footnotesize{The physical domain $\Omega$ including a triangulation on this domain.}}
  \label{fig:atlantic_domain}
\end{figure}
Notice that in the formulation of Definition \ref{def:linear_quadratic}, $C\in\B(Y,Z)$ is the injection operator. Furthermore, $M(\mu)\in\B(Z,Z^*)$ and $\Q(\mu)\in\B(U,U^*)$ are given by
\begin{align*}
\langle M(\mu)z,\tilde{z}\rangle_{Z^*Z} &= \int_\Omega z_1\tilde{z}_1\d{x} & z=(z_1,z_2),\tilde{z}=(\tilde{z}_1,\tilde{z}_2)\in Z,\\
  \langle \Q(\mu)u,\tilde{u}\rangle_{U^*U} &=  \alpha\int_\Omega u\tilde{u}\d{x} &u,\tilde{u}\in U,
\end{align*}
and $g=0$.%We shall try to solve this OCP by solving system (\ref{eqn:lin_qdr_opt}).

\paragraph{Truth approximation}
Taking the Finite Elements $X^1_h$ of (\ref{eqn:finite_elements}), we set $Y^{\nd_Y} = (X^1_h\times X^1_h)\cap Y$, $U^{\nd_U}=X^1_h$ and $P^{\nd_P}=Y^{\nd_Y}$. As is done in \cite{maria17}, we simply assume that $A^\nd(\mu)$ is invertible and thus the OCP is well-posed for $\P^{\bm{\mu}}$-a.e. $\mu\in\m$. 

\paragraph{Reduced Order Model}
We build a ROM by using a modification of the partitioned POD approach proposed in Subsection \ref{subsec:offline}. Writing $p=(w,q)$ for the adjoint solution, we perform five POD compressions, one for each of the $\vv$, $\rrho$, $u$, $w$ and $q$ components, and aggregate the spaces of state and adjoint. This leaves us with as system of size $9\nr\times9\nr$ to be solved. We then also construct ROMs by leaving out this aggregation step. In that case, we need only solve a system of size $5\nr\times 5\nr$. While we do not show that the reduced system is well-posed, notice that an efficient offline phase is ensured through the affine decomposition of Assumption \ref{ass:affine}, as $A$ is affine with $Q_A=3$ terms, and the terms $\Q,M,B,g$ and $\vv_d$ are parameter independent.

\paragraph{Results}
For a numerical implementation, the desired state $\vv_d$ is simulated by solving the quasi-geostrophic equation (\ref{eqn:geostrophic_state}) with $\mu_1=\mu_3=0$, $\mu_2=0.07^3$ and forcing term $u$ given by $(x_1,x_2)\mapsto -\sin(\pi x_2)$. Solving the truth problem for $\mu=(10^{-4}, 0.07^{3},0)$, we obtain the truth approximation $(\vv^\nd,\rho^\nd,u^\nd,w^\nd,q^\nd)$. % and the corresponding value $J^\nd=1.57203%4054906496\cdot 10^{-6}$
%\cdot 10^{-6}$. 
The dimension of the truth problem is $5813\times 5813$.
\newline
\newline
Afterwards, ROMs are built with $\nr=1,\ldots,20$ using Monte Carlo sampling to get a training set of size 100. Performing Galerkin projection of the truth solution onto the reduced basis space corresponding to $\nr=20$, a reduced \gcok{basis} approximation $(\vv_\nr,\rho_\nr,u_\nr,w_\nr,q_\nr)$ is obtained. The desired state $\vv_d$, solution component $\vv^\nd$ and pointwise difference $\vv_\nr-\vv^\nd$ for the parameter value $\mu=(10^{-4}, 0.07^{3},0)$ are shown in Figure \ref{fig:linear_geostrophic_offline}. %The reduced order approximation of the objective functional is accurate up to two decimal places.
\begin{figure}[!htb]
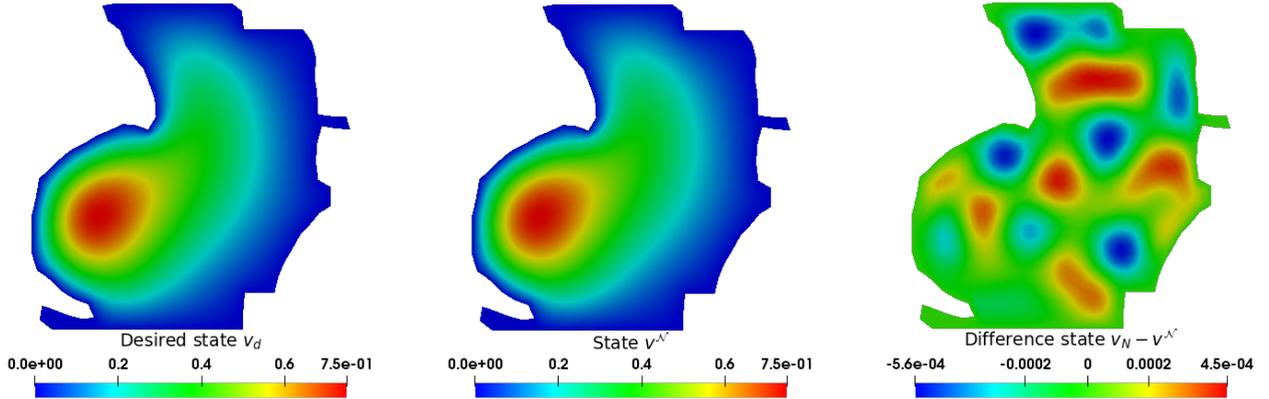

  \makebox[\textwidth][c]{
      \hspace{-27mm}
      \begin{subfigure}[b]{0.6\textwidth}
        \includegraphics[width=1.3\textwidth]{desired_state_v.png}
      \end{subfigure}
      \hspace{-40mm}
      \begin{subfigure}[b]{0.6\textwidth}
        \includegraphics[width=1.3\textwidth]{offline_state_v.png}
      \end{subfigure}
      \hspace{-40mm}
      \begin{subfigure}[b]{0.6\textwidth}
        \includegraphics[width=1.3\textwidth]{difference_state_v.png}
      \end{subfigure}
  }
  \caption{\footnotesize{Atlantic, linear case: \reviewerA{ROM obtained with Monte Carlo sampling from uniform distribution}: desired state $\vv_d$ (\textit{left}), truth state solution component $\vv^\nd$ (\textit{center}), and pointwise difference $\vv_\nr-\vv^\nd$, with $\nr=10$ (\textit{right}). The parameter value taken is $\mu=(10^{-4},0.07^{3},0)$.}}
  \label{fig:linear_geostrophic_offline}
\end{figure}
%\begin{figure}[H]
%  \makebox[\textwidth][c]{
%    \hspace{24mm}
%    \begin{subfigure}[b]{0.5\textwidth}
%      \includegraphics[width=0.7\textwidth]{relative_error_ypsi_Uniform_rules.png}
%    \end{subfigure}
%    \hspace{-28mm}
%    \begin{subfigure}[b]{0.5\textwidth}
%      \includegraphics[width=0.7\textwidth]{relative_error_ypsi_Uniform_rules_ucc.png}
%    \end{subfigure}
%    \hspace{-28mm}
%    \begin{subfigure}[b]{0.5\textwidth}
%      \includegraphics[width=0.7\textwidth]{relative_error_ypsi_Beta7575_rules.png}
%    \end{subfigure}
%  }
%  \makebox[\textwidth][c]{
%    \hspace{24mm}
%    \begin{subfigure}[b]{0.5\textwidth}
%      \includegraphics[width=0.7\textwidth]{relative_error_ypsi_LogUniform_rules.png}
%    \end{subfigure}
%    \hspace{-28mm}
%    \begin{subfigure}[b]{0.5\textwidth}
%      \includegraphics[width=0.7\textwidth]{relative_error_ypsi_LogUniform_rules_weighted.png}
%    \end{subfigure}
%    \hspace{-28mm}
%    \begin{subfigure}[b]{0.5\textwidth}
%      \includegraphics[width=0.7\textwidth]{relative_error_ypsi_Beta51_rules.png}
%    \end{subfigure}
%  }
%  \caption{Atlantic, linear case: average logarithmic relative errors in the $v$-component in the deterministic case (\textit{top left/center}), under a Beta(75,75) (\textit{top right}), Loguniform (\textit{bottom left/center}), and Beta(5,1) distribution (\textit{bottom right}), as a function of $\nr$. Various quadrature rules are compared.}
%  \label{fig:geostrophic_errors_rules}
%\end{figure}

Next, we apply the ROMs constructed with the aggregation approach to compute reduced \gcok{basis} approximations for parameter values in a testing set of size 100.  A plot of the average logarithmic errors over this testing set are shown in Figure \ref{fig:geostrophic_errors_uniform_mc}.
\begin{figure}[!htb]
\begin{center}
    \begin{subfigure}[b]{0.4\textwidth}
      \includegraphics[width=1.0\textwidth]{relative_error_ypsi_Uniform_mc.png}
    \end{subfigure}
    %\hspace{-28mm}
    \begin{subfigure}[b]{0.4\textwidth}
      \includegraphics[width=1.0\textwidth]{relative_error_u_Uniform_mc.png}
    \end{subfigure}

    \begin{subfigure}[b]{0.4\textwidth}
      \includegraphics[width=1.0\textwidth]{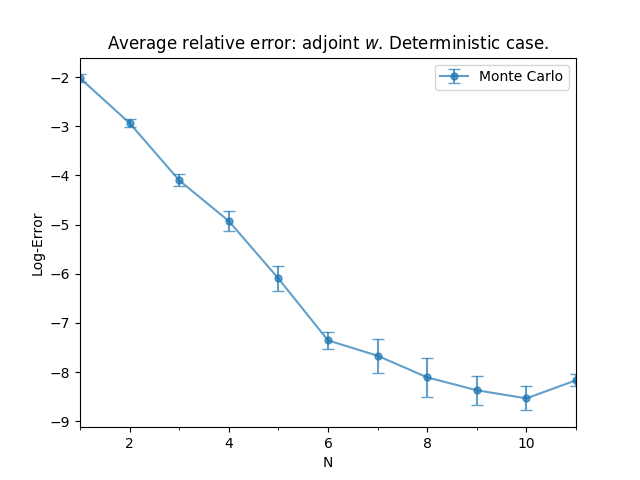}
    \end{subfigure}
    %\hspace{-28mm}
    \begin{subfigure}[b]{0.4\textwidth}
      \includegraphics[width=1.0\textwidth]{relative_error_output_Uniform_mc.png}
    \end{subfigure}
    \caption{\footnotesize{Atlantic, linear case: average logarithmic relative errors for the $\vv$-component of state (\textit{top left}), control (\textit{top right}), $w$-component of adjoint (\textit{bottom left}) and output (\textit{bottom right}) \reviewerA{in the deterministic case}, as a function of $\nr$. \reviewerA{ROM obtained via Monte Carlo sampling.} }}
  \label{fig:geostrophic_errors_uniform_mc}
\end{center}
\end{figure}
The optimal number of basis functions for the control space is larger than that for the other spaces. Still, a low number of basis vectors is required to ensure small relative errors. We compare these results with ROMs obtained with different quadrature rules in Figure \ref{fig:geostrophic_errors_uniform_rules_comparison}. Each ROM is constructed twice, once with and once without the aggregation step. All four rules generate a training set of size 100, and only the error for the $\vv$-component of the state is shown. The Monte Carlo, Pseudo Random and Gaussian quadrature perform similarly, and notably much better than the Clenshaw-Curtis quadrature rule that is implemented. Furthermore, we note that the ROMs constructed by skipping the aggregation step, actually reach a higher accuracy. While they also require a larger number $\nr$ of basis functions to be retained, the fact that aggregation is skipped means that only a system of dimension $5\nr\times 5\nr$ instead of $9\nr\times 9\nr$ must be solved online. For example, the most accurate ROM constructed without the aggregation approach is obtained with the Pseudo-Random Sampler with $\nr=19$, which leads to a system to be solved online of size $95\times 95$. The optimal ROM with an aggregation approach is constructed with the Monte Carlo sampler with $\nr=11$, so that the online system is of size $99\times 99$. Furthermore, the optimal ROM constructed without the aggregation approach, is on average more accurate by about two orders of magnitude.
\newline
\newline
In Figure \ref{fig:geostrophic_errors_beta_rules_comparison} (\textit{left}) and Figure \ref{fig:geostrophic_errors_loguniform_rules_comparison} (\textit{left}) also ROMs for a Beta($75,75$), Beta($5,1$) and Loguniform distribution are considered. Each is constructed by aggregating state and adjoint. For the Beta(75,75) distribution, the Pseudo Random sampler picks out an extra useful direction in the solution manifold of the $\vv$-component. With only $\nr=5$ very high accuracy is obtained. 
For the Beta(5,1) distribution, that puts more probability mass on the larger parameter values, the Monte Carlo and Pseudo Random samplers seem to be preferable over Gaussian quadrature. For both distributions, the difference between the quadratures is small. Interestingly, the same can not be concluded in the Loguniform case, which emphasizes small parameter values, as the performance of both the Clenshaw-Curtis and the Gaussian quadrature is poor, and their accuracies vary more extensively over the testing set.
\newline
\newline
The same distributions are considered in Figures \ref{fig:geostrophic_errors_beta_rules_comparison} (\textit{right}) and \ref{fig:geostrophic_errors_loguniform_rules_comparison} (\textit{right}), but this time the corresponding ROMs are constructed skipping the aggregation step. Let us compare the optimal ROMs contructed with and without the aggregation procedure. For the Beta($75,75$) distribution, the optimal ROM with aggregation is constructed with a Pseudo Random sampler for $\nr=6$, and without aggregation with the Monte Carlo sampler for $\nr=8$. For the Beta($5,1$) distribution, these are respectively the Monte Carlo sampler with $\nr=5$ and the Monte Carlo sampler with $\nr=10$. In both cases, the accuracy of the ROM obtained without aggregation is at least as high as for the ROM obtained with aggregation.
\newline
\newline
For the Loguniform distribution, the optimal ROMs are the Monte Carlo rules with $\nr=35$ and Pseudo Random sampler with $\nr=40$, but we do note that in this case we stopped the simulations at $\nr=40$ so that higher accuracy can possibly be obtained. 
\newline
\newline
Note that the errors corresponding to the aggregation approach do decay more monotonously. Non monotonous decay can occur in general, due to the fact that the Galerkin Projection is not an orthogonal projection but a skewed projection.

\begin{figure}[!htb]
\begin{center}
  \makebox[0.4\textwidth]{\centering \textbf{With aggregation}}
  \makebox[0.4\textwidth]{\textbf{Without aggregation}}
    \begin{subfigure}[b]{0.4\textwidth}
      \includegraphics[width=1.0\textwidth]{relative_error_ypsi_Uniform_rules.png}
    \end{subfigure}
    \begin{subfigure}[b]{0.4\textwidth}
      \includegraphics[width=1.0\textwidth]{relative_error_ypsi_Uniform_rules_no_agg.png}
    \end{subfigure}

    \begin{subfigure}[b]{0.4\textwidth}
      \includegraphics[width=1.0\textwidth]{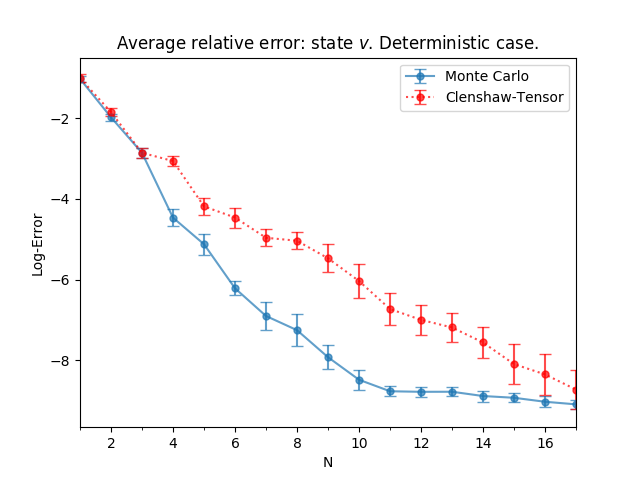}
    \end{subfigure}
    \begin{subfigure}[b]{0.4\textwidth}
      \includegraphics[width=1.0\textwidth]{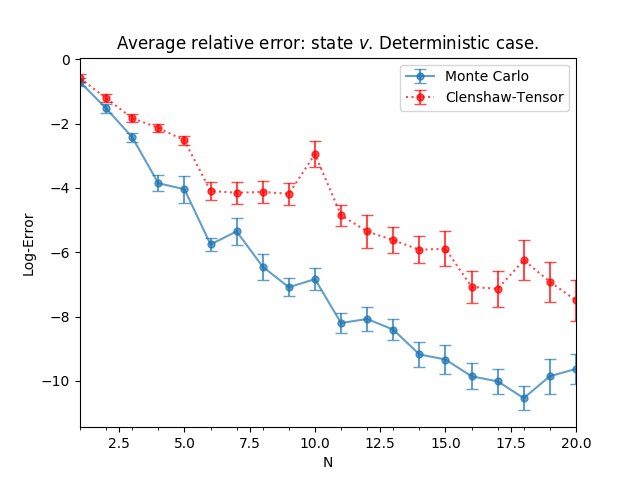}
    \end{subfigure}

    \caption{\footnotesize{Atlantic, linear case: average logarithmic relative errors in the $\vv$-component in the deterministic case, for ROMs constructed with aggregation (\textit{left}) and without aggregation (\textit{right}), as a function of $\nr$. Various quadrature rules are compared.}}
  \label{fig:geostrophic_errors_uniform_rules_comparison}
\end{center}
\end{figure}

\begin{figure}[!htb]
\begin{center}
  \makebox[0.4\textwidth]{\textbf{With aggregation}}
  \makebox[0.4\textwidth]{\textbf{Without aggregation}}

    \begin{subfigure}[b]{0.4\textwidth}
      \includegraphics[width=1.0\textwidth]{relative_error_ypsi_Beta7575_rules.png}
    \end{subfigure}
    \begin{subfigure}[b]{0.4\textwidth}
      \includegraphics[width=1.0\textwidth]{relative_error_ypsi_Beta7575_rules_no_agg.png}
    \end{subfigure}

    \begin{subfigure}[b]{0.4\textwidth}
      \includegraphics[width=1.0\textwidth]{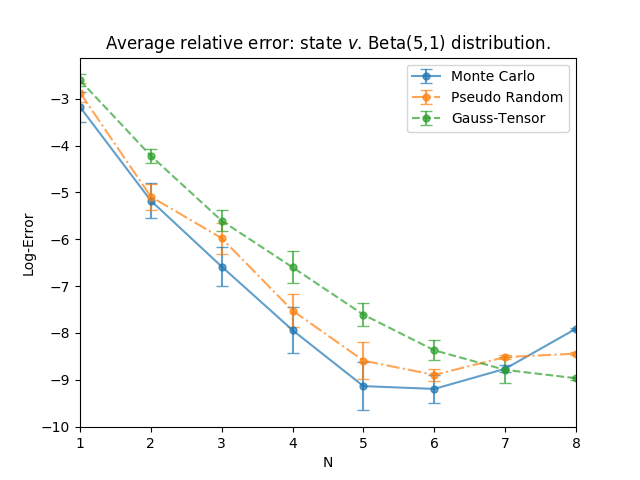}
    \end{subfigure}
    \begin{subfigure}[b]{0.4\textwidth}
      \includegraphics[width=1.0\textwidth]{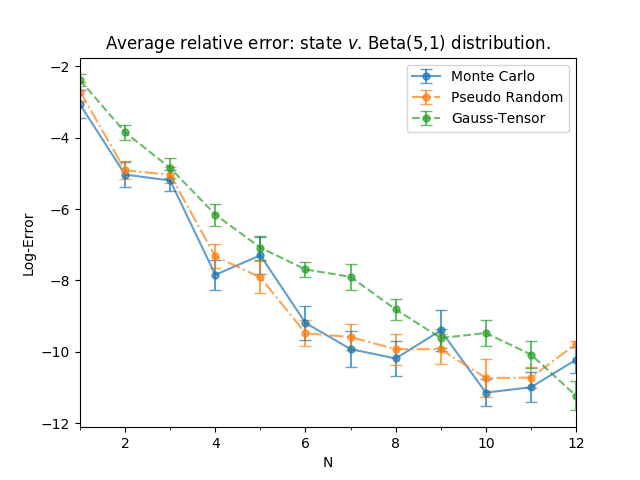}
    \end{subfigure}

  \caption{\footnotesize{Atlantic, linear case: average logarithmic relative errors in the $\vv$-component under a Beta$(75,75)$ distribution (\textit{top}) and Beta$(5,1)$ distribution (\textit{bottom}), for ROMs constructed with aggregation (\textit{left}) and without aggregation (\textit{right}), as a function of $\nr$. Various quadrature rules are compared.}}
  \label{fig:geostrophic_errors_beta_rules_comparison}
\end{center}
\end{figure}

\begin{figure}[!htb]
\begin{center}
  \makebox[0.4\textwidth]{\textbf{With aggregation}}
  \makebox[0.4\textwidth]{\textbf{Without aggregation}}

    \begin{subfigure}[b]{0.4\textwidth}
      \includegraphics[width=1.0\textwidth]{relative_error_ypsi_LogUniform_rules.png}
    \end{subfigure}
    \begin{subfigure}[b]{0.4\textwidth}
      \includegraphics[width=1.0\textwidth]{relative_error_ypsi_LogUniform_rules_no_agg.png}
    \end{subfigure}

    \begin{subfigure}[b]{0.4\textwidth}
      \includegraphics[width=1.0\textwidth]{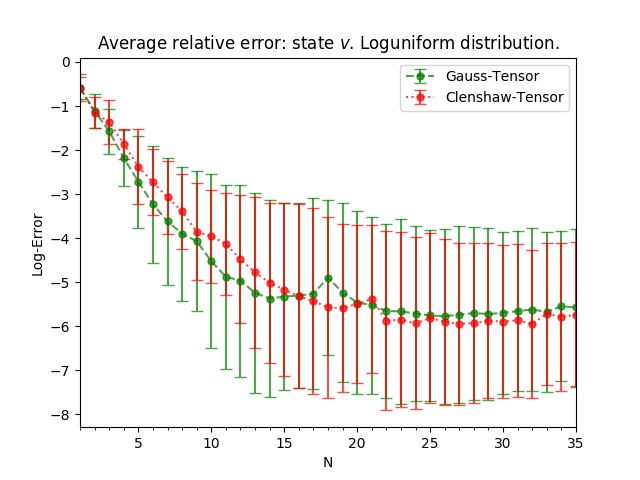}
    \end{subfigure}
    \begin{subfigure}[b]{0.4\textwidth}
      \includegraphics[width=1.0\textwidth]{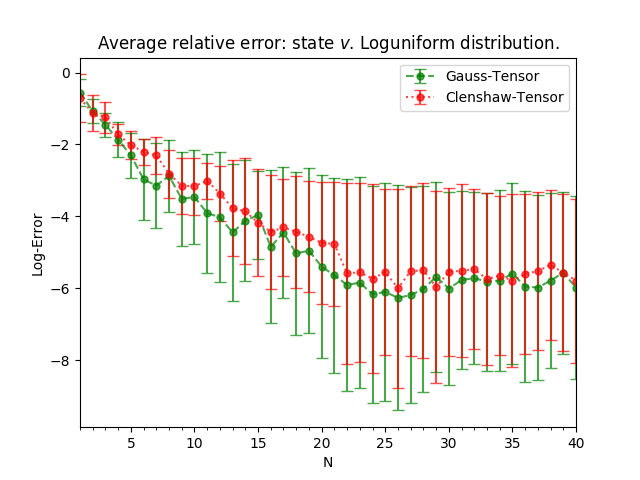}
    \end{subfigure}

    \caption{\footnotesize{Atlantic, linear case: average logarithmic relative errors in the $\vv$-component under a Loguniform distribution, for ROMs constructed with aggregation (\textit{left}) and without aggregation (\textit{right}), as a function of $\nr$. Various quadrature rules are compared.}}
  \label{fig:geostrophic_errors_loguniform_rules_comparison}
\end{center}
\end{figure}

Finally, we present the speedup-index in the deterministic case with Monte Carlo sampling in Table \ref{tab:geostrophic_speedup}. With $\nr=12$, an average speedup of over 70 times is achieved using a ROM constructed with the aggregation step. This speedup is obtained with a ROM of $\nr=20$ in case the aggregation step is skipped. Notice, however, that the deviations in the speedup index are larger for ROMs constructed without the aggregation step.

  \begin{table}[!htb]
    \makebox[0.5\textwidth]{\textbf{With aggregation}}
    \makebox[0.5\textwidth]{\textbf{Without aggregation}}
\makebox[0.5\textwidth]{
  \centering
  \begin{tabular}{|r|r|r|r|c|}
    \hline
    $\nr$ & \gcok{average} & min & max & deviation \\
    \hline
    2 & 99.6 & 55.9 & 129.9 & 9.9 \\
    4 & 95.5 & 63.6 & 123.0 & 8.9 \\
    8 & 84.9 & 61.8 & 110.7 & 7.9 \\
    12 & 72.8 & 54.7 & 94.6 & 6.3 \\
    16 & 58.3 & 41.4 & 75.4 & 5.7 \\
    20 & 45.7 & 33.0 & 58.9 & 4.3 \\
    \hline
  \end{tabular}
}
\makebox[0.5\textwidth]{
  \centering
  \begin{tabular}{|r|r|r|r|c|}
    \hline
    $\nr$ & \gcok{average} & min & max & deviation \\
    \hline
    2 & 99.8 & 51.4 & 145.9 & 17.1 \\
    4 & 98.4 & 60.8 & 144.9 & 14.8 \\
    8 & 96.0 & 57.6 & 128.4 & 14.2 \\
    12 & 83.9 & 48.2 & 125.1 & 14.3 \\
    16 & 80.5 & 48.4 & 114.5 & 12.3 \\
    20 & 72.2 & 37.6 & 103.0 & 12.1 \\
    \hline
  \end{tabular}
}
\caption{\footnotesize{Atlantic, linear case: sample \gcok{average}, minimal value, maximal value, and sample standard deviation of speedup-index obtained with a testing set of size 100. Deterministic case. ROM obtained by Monte Carlo sampling, with aggregation (\textit{left}) and without aggregation (\textit{right}).}}
  \label{tab:geostrophic_speedup}
\end{table}

\subsection{Nonlinear quasi-geostrophic equation on the Atlantic Ocean}
\graphicspath{{images/geostrophic_nonlinear/}}
Having investigated the performance of a ROM for the linearized version of the quasi-geostrophic state equation, we now relax the condition $\mu_3=0$ and thus allow the nonlinearity to enter the system. Apart from this modification, the control problem is the same as in Subsection \ref{subsec:linear_atlantic}.

\paragraph{Optimal Control for problems with a nonlinear state equation}
By Theorem \ref{thm:constrained_opt_eqns}, the OCP for $\mu$ in a set of probability one can be solved by solving the system (\ref{eqn:general_optimizing_equations}) for that $\mu$, if $(D\!A(\mu))y\in \B(Y,P^*)$ has a bounded inverse for every $y\in Y$, but we do not dwell on this assumption to hold. Writing $y=(\vv,\rrho)$ and taking $(\tilde{\vv},\tilde{\rrho})\in Y$, $p=(w,q)\in P$, we have
\begin{align*}
  \langle ((D\!A(\mu))y)(\tilde{\vv},\tilde{\rrho}),p \rangle_{P^*P} &= \langle A_0(\mu)(\tilde{\vv},\tilde{\rrho}),(w,q)\rangle_{P^*P} \\&+ \mu_3 \int_{\Omega}^{}\F(\tilde{\vv},\rrho)w\d{x} + \mu_3 \int_{\Omega}^{}\F(\vv,\tilde{\rrho})w\d{x}
\end{align*}
If we denote the adjoint variable by $p(\mu)=(w(\mu),q(\mu))$, the adjoint equation, the second equation of (\ref{eqn:general_optimizing_equations}), in this case reads (suppressing the injection $H^1_0(\Omega)\hookrightarrow L^2(\Omega)$ in the last term)
\begin{align}
  \begin{split}
    0 &= \langle A_0(\mu) (\tilde{\vv},\tilde{\rrho}),( w (\mu), q(\mu))\rangle_{P^*P} - \mu_3\int_{\Omega}^{}\F(\vv(\mu),w(\mu))\tilde{\rrho}\d{x} \\&- \mu_3\int_{\Omega}^{}\F(w(\mu),\rrho(\mu))\tilde{\vv}\d{x} + \langle M(\mu) (\vv(\mu) - \vv_d,\rrho(\mu) ),(\tilde{\vv},\tilde{\rrho}) \rangle_{Z^*Z}\reviewerB{,}
  \end{split}
    \label{eqn:geostrophic_adjoint_equation}
\end{align}
while the state equation and optimality equation are still of the same form as those in (\ref{eqn:gen_lin_qdr_opt}).

\paragraph{Truth formulation and Reduced Order Model}
The truth formulation for this nonlinear problem is obtained no differently than in the linear case. The spaces $Y^{\nd_Y},U^{\nd_U}$ and $P^{\nd_P}$ are as in the linear case. The truth formulation then is formed by solving (\ref{eqn:geostrophic_adjoint_equation}) together with the last two equations of (\ref{eqn:lin_qdr_opt}), with $Y,U,P$ replaced by $Y^{\nd_Y},U^{\nd_U},P^{\nd_P}$. Being a nonlinear system, it cannot be solved at once. An iterative procedure can be used to find a solution, and we choose to employ Newton's method to solve the truth problem. While we do not theoretically show well-posedness of the truth problem, we decide to remain in a \textit{mild nonlinear setting} by the choice of parameter space, see \cite{crisciani2013}. This has always led to a convergent Newton solver. For higher values of the nonlinear parameter $\mu_3$, one shall need to stabilize the system at hand.
\newline
A ROM is constructed using the same methods as for the linear case. Again, well-posedness is assumed and a Newton iteration procedure is used to solve the reduced system. %All involved operators satisfy the affine decomposition, so that Assumption \ref{ass:affine} is satisfied.

\paragraph{Results}
This time the desired state $\vv_d$ is simulated by solving the quasi geostrophic equation (\ref{eqn:geostrophic_state}) with $\mu=(0,0.07^3,0.07^2)$ and forcing term $(x_1,x_2)\mapsto -\sin(\pi x_2)$. Let us first assume the deterministic setting.
The truth problem is solved for the parameter value $\mu=(10^{-4}, 0.07^{3},0.045^2)$. %The value $J^\nd=2.043478\cdot 10^{-6}$. %2.0434773126097594\cdot 10^{-6}$.
% The system of the truth problem has dimension $\nr\times \nr = 5953\times 5953$.
\newline
\newline
Assuming $\bm{\mu}$ is uniformly distributed, we construct a ROM for $\nr=1,\ldots,15$, with the aggregation approach. For $\nr=15$ a reduced solution component $\vv_\nr$ was computed for $\mu=(10^{-4}, 0.07^{3},0.045^2)$. The desired state $\vv_d$, truth solution component $\vv^\nd$ and pointwise difference $\vv_\nr-\vv^\nd$ for this parameter value are shown in Figure \ref{fig:nonlinear_geostrophic_offline}. Furthermore, the objective functional is approximated with an accuracy of 13 digits.
\begin{figure}[!htb]
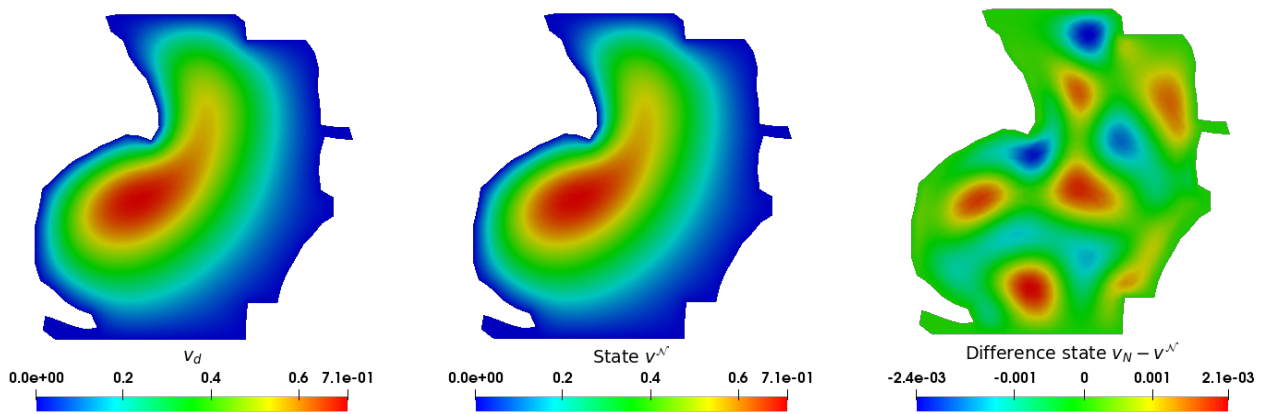

  \makebox[\textwidth][c]{
      \hspace{-27mm}
      \begin{subfigure}[b]{0.6\textwidth}
        \includegraphics[width=1.3\textwidth]{desired_state_v.png}
      \end{subfigure}
      \hspace{-40mm}
      \begin{subfigure}[b]{0.6\textwidth}
        \includegraphics[width=1.3\textwidth]{offline_state_v.png}
      \end{subfigure}
      \hspace{-40mm}
      \begin{subfigure}[b]{0.6\textwidth}
        \includegraphics[width=1.3\textwidth]{difference_state_v.png}
      \end{subfigure}
  }
  \caption{\footnotesize{Atlantic, nonlinear case: \reviewerA{ROM obtained with Monte Carlo sampling from uniform distribution} desired state $\vv_d$ (\textit{left}), truth solution component $\vv^\nd$ (\textit{center}), and pointwise difference $\vv_\nr-\vv^\nd$ (\textit{right}), with $\nr=10$. The parameter value taken is $\mu=(10^{-4},0.07^{3},0.045^2)$.}}
  \label{fig:nonlinear_geostrophic_offline}
\end{figure}
The average logarithmic error of the output and $\vv$-component are displayed in Figure \ref{fig:nonlinear_geostrophic_errors_comparison} (\textit{top left}). Once more exponential decay can be observed. In Figure \ref{fig:nonlinear_geostrophic_errors_comparison} (\textit{center left, bottom left}) a Beta(75,75) distribution and a Loguniform distribution on each of the three parameter domains is considered. In each case, different quadrature rules are compared, and ROMs are constructed with the aggregation approach. The Beta(75,75) distribution again reduces the complexity of the model, which results in needing only $\nr=5$ to reconstruct the discrete solution manifold well. The Loguniform distribution achieves similar accuracy after $\nr=25$. It should be noted that the Monte Carlo and Pseudo Random rules taken use a training set of size 100, while the Gaussian and Clenshaw-Curtis rules use a training set of size 125.
\newline
\newline
Despite working in the mild nonlinear setting, for some parameter value the Newton Iteration procedure has diverged for the ROM with $\nr=24$ and a Gauss sampler. This results in the large deviation that is observed.

\begin{figure}[!htb]
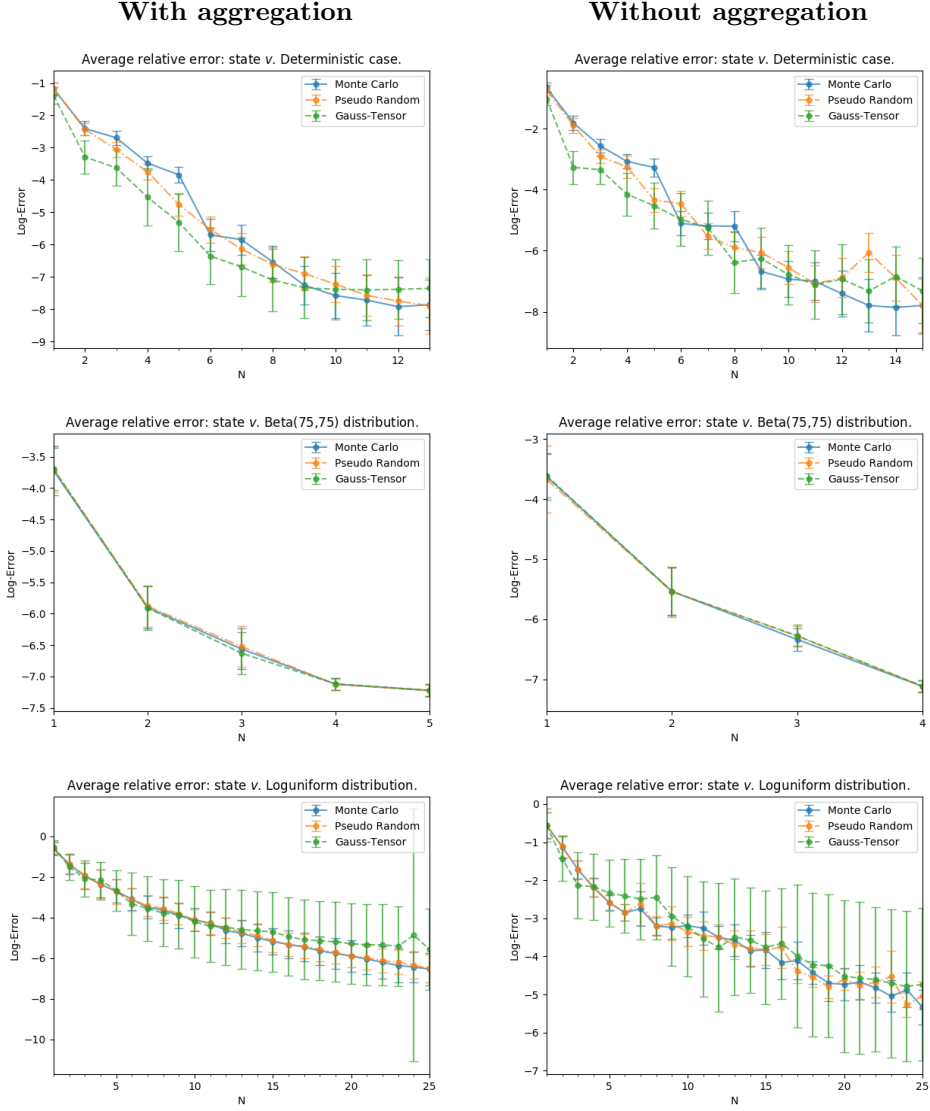

\begin{center}
  \makebox[0.4\textwidth]{\textbf{With aggregation}}
  \makebox[0.4\textwidth]{\textbf{Without aggregation}}

    \begin{subfigure}[b]{0.4\textwidth}
      \includegraphics[width=1.0\textwidth]{relative_error_ypsi_Uniform_rules.png}
    \end{subfigure}
    \begin{subfigure}[b]{0.4\textwidth}
      \includegraphics[width=1.0\textwidth]{relative_error_ypsi_Uniform_rules_no_agg.png}
    \end{subfigure}

    \begin{subfigure}[b]{0.4\textwidth}
      \includegraphics[width=1.0\textwidth]{relative_error_ypsi_Beta7575_rules.png}
    \end{subfigure}
    \begin{subfigure}[b]{0.4\textwidth}
      \includegraphics[width=1.0\textwidth]{relative_error_ypsi_Beta7575_rules_no_agg.png}
    \end{subfigure}

    \begin{subfigure}[b]{0.4\textwidth}
      \includegraphics[width=1.0\textwidth]{relative_error_ypsi_LogUniform_rules.png}
    \end{subfigure}
    \begin{subfigure}[b]{0.4\textwidth}
      \includegraphics[width=1.0\textwidth]{relative_error_ypsi_LogUniform_rules_no_agg.png}
    \end{subfigure}
  %}
    \caption{\footnotesize{Atlantic, nonlinear case: average logarithmic relative errors for the state $\vv$-component in the deterministic case (\textit{top}), under a Beta($75,75$) distribution (\textit{center}) and under a Loguniform distribution (\textit{bottom}), for ROMs constructed with aggregation (\textit{left}) and without aggregation (\textit{right}), as a function of $\nr$. Various quadrature rules are compared.}}
  \label{fig:nonlinear_geostrophic_errors_comparison}
\end{center}
\end{figure}

The same ROMs are constructed leaving out the aggregation step, and the corresponding results are shown in Figure \ref{fig:nonlinear_geostrophic_errors_comparison} (\textit{right}). As in the example of Subsection \ref{subsec:linear_atlantic}, we can conclude that without the need of the aggregation step, the obtained results are of an accuracy that is at least as high as in an aggregation approach, while saving online computation time in general.

In Table \ref{tab:nonlinear_geostrophic_speedup} the speedup-index for the Monte Carlo Sampler in the deterministic setting is shown. The speedups are small due to the nonlinearity, as the system that needs to be solved online still needs to be assembled in a number of computations that depends on $\nd$. Furthermore, the Newton iteration used to solve this system can take many steps to converge. While the ROMs are still effective, this lack of efficiency renders the ROM less useful in this application. Nevertheless, there are strategies which can overcome this issue, the interested reader may refer to see \cite{barrault04}.

\begin{table}[H]
  \makebox[0.5\textwidth]{\textbf{With aggregation}}
  \makebox[0.5\textwidth]{\textbf{Without aggregation}}
  \makebox[0.5\textwidth]{
  \centering
  \begin{tabular}{|r|r|r|r|c|}
    \hline
    $\nr$ & \gcok{average} & min & max & deviation \\
    \hline
    2 & 1.8 & 1.3 & 2.3 & 0.20 \\
    3 & 1.7 & 1.3 & 2.1 & 0.20 \\
    6 & 1.3 & 1.0 & 1.7 & 0.18 \\
    9 & 1.0 & 0.80 & 1.4 & 0.14 \\
    12 & 0.83 & 0.60 & 1.1 & 0.12 \\
    15 & 0.66 & 0.45 & 0.86 & 0.097 \\
    \hline
  \end{tabular}
  }
  \makebox[0.5\textwidth]{
  \centering
  \begin{tabular}{|r|r|r|r|c|}
    \hline
    $\nr$ & \gcok{average} & min & max & deviation \\
    \hline
    2 & 1.8 & 1.4 & 2.9 & 0.25 \\
    3 & 1.8 & 1.1 & 3.1 & 0.27 \\
    6 & 1.6 & 1.0 & 2.6 & 0.26 \\
    9 & 1.4 & 0.9 & 2.4 & 0.25 \\
    12 & 1.3 & 0.83 & 2.0 & 0.20 \\
    15 & 1.1 & 0.75 & 1.9 & 0.18 \\
    \hline
  \end{tabular}
  }
  \caption{\footnotesize{Atlantic, nonlinear case: sample \gcok{average}, minimal value, maximal value, and sample standard deviation of speedup-index obtained with a testing set of size 100. Deterministic case. ROM obtained by Monte Carlo sampling, with aggregation (\textit{left}) and without aggregation (\textit{right}).}}
  \label{tab:nonlinear_geostrophic_speedup}
\end{table}

\section{Conclusions}
\label{sec:conclusions}
As an alternative to the Lagrangian approach applicable to linear-quadratic Optimal Control Problems with full-admissibility, we have studied OCPs which may have admissibility constraints or which are not of linear-quadratic nature by using the adjoint approach. As a by-product we have concluded that to establish well-posedness of OCPs and their approximations, the use of aggregate spaces can be avoided, because the fundamental requirement is invertibility of the reduced state operator. Having said this, in the coercive case the use of aggregate spaces is still useful, because it guarantees this invertibility.
\newline
We have also studied OCPs with random inputs, and used the weighted POD algorithm to construct ROMs for OCPs with full admissibility. The weighted POD has enabled us to incorporate information on parameter distributions, that originate from, for example, uncertainties involved in experimental measurements. The ROMs then allowed us to accurately and efficiently compute approximations to linear-quadratic OCPs with full admissibility in several marine science scenarios. We have also embedded a scenario in which the governing equation is the nonlinear single-layer \gcok{steady} quasi-geostrophic equation in an Uncertainty Quantification context.
\newline
In \gcok{our} scenarios for which the governing equation was not elliptic, we have numerically confirmed that it is not \reviewerB{needed to} use aggregate spaces. Indeed, the ROMs that were constructed without this aggregation performed at least as well, and in most cases resulted in a speedup through a smaller system that has to be solved in the online phase. Furthermore, we have considered various types of distributions on the parameter space. In each case, the constructed ROMs effectively incorporated this information. As the weighted POD algorithm depends on a quadrature rule, we have explored several implementations of quadrature rules. No numerical evidence that one specific rule should be preferred over the others has been found, although the chosen implementation of the Clenshaw-Curtis rule should be avoided. Henceforth, one might as well use the Monte Carlo sampler, as it is the simplest one.

\section*{Acknowledgements}
We acknowledge the funding granted via the European Eramus+ project and the support by European Union Funding for Research and Innovation -- Horizon 2020 Program -- in the framework of European Research Council Executive Agency: Consolidator Grant H2020 ERC CoG 2015 AROMA-CFD project 681447 ``Advanced Reduced Order Methods with Applications in Computational Fluid Dynamics''. We also acknowledge the INDAM-GNCS project ``Advanced intrusive and non-intrusive model order reduction techniques and applications'' and the PRIN 2017  ``Numerical Analysis for Full and Reduced Order Methods for the efficient and accurate solution of complex systems governed by Partial Differential Equations'' (NA-FROM-PDEs).
The computations in this work have been performed with RBniCS \cite{rbnics} library, developed at SISSA mathLab, which is an implementation in FEniCS \cite{fenics} of several reduced order modelling techniques; we acknowledge developers and contributors to both libraries.

\addcontentsline{toc}{chapter}{Bibliography}
\bibliography{mybibitems.bib}
\bibliographystyle{apalike}

\end{document}